\documentclass[12pt]{article}

\usepackage{amsmath}    
\usepackage{graphicx}   
\usepackage{verbatim}   
\usepackage{color}      
\usepackage{subfigure}  
\usepackage{amssymb}
\usepackage{cite}
\usepackage{amsthm}
\usepackage[utf8]{inputenc}  

\usepackage[linesnumbered,ruled,vlined]{algorithm2e}
\SetKwInput{KwInput}{Input}                
\SetKwInput{KwOutput}{Output}              
\usepackage{algcompatible}

\usepackage{adjustbox}

\newtheorem{theorem}{Theorem}[section]
\newtheorem{proposition}[theorem]{Proposition}

\newtheorem{lemma}[theorem]{Lemma}

\newtheoremstyle{case}{}{}{}{}{}{:}{ }{}
\theoremstyle{case}
\newtheorem{case}{\textbf{Case}}
\newcommand{\setcaseprenumber}[1]{%
	\setcounter{case}{0}%
}

\numberwithin{subcase}{case}

\newcommand{\R}{\mathbb R}
\newcommand{\N}{\mathbb N}

\def\cp{\,\square\,}

\DeclareMathOperator{\sg}{\rm sg}
\DeclareMathOperator{\sge}{\rm sg_e}
\DeclareMathOperator{\diam}{\rm diam}
\usepackage{tikz}
\usepackage{ifthen}

\newcommand\mreza[2]{
	\pgfmathsetmacro{\n}{#1}
	\pgfmathsetmacro{\m}{#2}
	\ifthenelse{\m=0}{
		\foreach \i in {1,...,\n}{
			\draw (\i,1) -- (\i,2);
			\draw (\i,3) -- (\i,4);
			\draw[color=white, text=black] (\i,2.6) node {\tiny \vdots};
		}
		\foreach \j in {1,...,4}{
			\draw (1,\j) -- (\n,\j);
		}
		\foreach \i in {1,...,\n}{
			\foreach \j in {1,...,4}{
				\draw[fill=white] (\i,\j) circle(5pt);
		}}
	}
	{
		\ifthenelse{\n = 0}{
			\foreach \i in {1,...,3}{
				\draw (\i,1) -- (\i,\m);
			}
			\draw (5,1) -- (5,\m);
			\draw (6,1) -- (6,\m);
			\foreach \j in {1,...,\m}{
				\draw (1,\j) -- (3.5,\j);
				\draw (4.5,\j) -- (6,\j);
				\draw[color=white, text=black] (4,\j) node {\tiny \dots};
			}
			\foreach \j in {1,...,\m}{
				\draw[fill=white] (1,\j) circle(5pt);
				\draw[fill=white] (2,\j) circle(5pt);
				\draw[fill=white] (3,\j) circle(5pt);
				\draw[fill=white] (5,\j) circle(5pt);
				\draw[fill=white] (6,\j) circle(5pt);
			}
		}{
			\foreach \i in {1,...,\n}{
				\draw (\i,1) -- (\i,\m);
			}
			\foreach \j in {1,...,\m}{
				\draw (1,\j) -- (\n,\j);
			}
			\foreach \i in {1,...,\n}{
				\foreach \j in {1,...,\m}{
					\draw[fill=white] (\i,\j) circle(5pt);
			}}
		}
	}
}

\newcommand\mrezacrtkano[2]{
	\pgfmathsetmacro{\n}{#1}
	\pgfmathsetmacro{\m}{#2}
	\foreach \i in {1,...,\n}{
		\draw[densely dotted] (\i,1) -- (\i,\m);
	}
	\foreach \j in {1,...,\m}{
		\draw[densely dotted] (1,\j) -- (\n,\j);
	}
	\foreach \i in {1,...,\n}{
		\foreach \j in {1,...,\m}{
			\draw[fill=white] (\i,\j) circle(5pt);
		}
	}
}

\begin{document}
	
\tikzstyle{every node}=[circle, draw, fill=white,
inner sep=0pt, minimum width=10pt]

\title{Strong edge geodetic problem on grids}
\author{Eva Zmazek}

\date{}

\maketitle
\vspace{-0.8 cm}
\begin{center}
	Faculty of Mathematics and Physics, University of Ljubljana, Slovenia\\
	{\tt eva.zmazek@fmf.uni-lj.si}
\end{center}


\begin{abstract}
Let $G=(V(G),E(G))$ be a simple graph. A set $S \subseteq V(G)$ is a strong edge geodetic set if there exists an assignment of exactly one shortest path between each pair of vertices from $S$, such that these shortest paths cover all the edges $E(G)$. The cardinality of a smallest strong edge geodetic set is the strong edge geodetic number $\sge(G)$ of $G$.
In this paper, the strong edge geodetic problem is studied on the Cartesian product of two paths. The exact value of the strong edge geodetic number is computed for $P_n \cp P_2$, $P_n \cp P_3$ and $P_n \cp P_4$. Some general upper bounds for $\sge(P_n \cp P_m)$ are also proved.
\end{abstract}


\medskip
\noindent{\bf AMS Subj.\ Class.: 05C12, 05C70}

\section{Introduction}

Different covering problems with shortest paths were studied in literature. For example, the geodetic problem was introduced in $1993$ in \cite{harary1993tgnoag} and its edge version in $2007$ in \cite{santhakumaran2007egnoag}. In $2016$, the strong geodetic problem was introduced, the seminal paper \cite{manuel2020sgpin} being published only recently. Since then, a lot of work was done on the strong geodetic problem. 

The exact value of the strong geodetic number was computed for different families of graphs. For example, for complete bipartite graphs $K_{n,m}$ it was first computed for cases when $n=m$, and for $n \gg m$ in \cite{irsic2018sgnocbgagwsd}, and later in the general case in \cite{gledel2020sgnocbgcgah}. In \cite{irsic2019sgpocmg}, the strong geodetic number was computed for some balanced multipartite complete graphs and it was shown that computing the strong geodetic number of general complete multipartite graphs is NP-complete. The exact strong geodetic number was also computed for crown graphs $S_0^n$ in \cite{gledel2020sgnocbgcgah}, Hamming graphs $K_m \cp K_n$ in \cite{irsic2018sgpocpog}, Cartesian products $K_{1,n} \cp P_l$ in \cite{irsic2018sgpocpog}, thin ($n \gg m$) grids $P_n \cp P_m$, and thin ($n \gg m$) cylinders $P_n \cp C_m$ in \cite{klavzar2018sgpigla}, and $i$-level complete Appolonian networks $A(i)$ in \cite{manuel2020sgpin}.

Several general bounds on the strong geodetic number were given using different graphical invariants. In \cite{irsic2018sgnocbgagwsd}, bounds on $\sge(G)$ were given depending on the diameter $\diam(G)$ of $G$. In \cite{wang2020sgnogac}, an upper bound on the strong geodetic number was given using the connectivity number. In \cite{manuel2020sgpin}, lower and upper bounds were given using isometric path number. Gledel at al. \cite{gledel2019sgcacpg} proved an upper bound on the strong geodetic number of Cartesian product graphs using the so-called strong geodetic core number. In \cite{gledel2020sgnocbgcgah}, an upper bound on the strong edge geodetic number was given for hypercubes. A general upper bound was in \cite{irsic2018sgpocpog} computed for the strong geodetic number of Cartesian product graphs and an upper bound for the strong geodetic number of the Cartesian product of a path with an arbitrary graph was computed in \cite{klavzar2018sgpigla}.

Using reduction from NP-completeness of dominating set problem, Manuel et al. in \cite{manuel2020sgpin} proved that the strong geodetic problem is NP-complete. On the positive side Mezzini in \cite{mezzini2020aoafctsgniog} gave a polynomial algorithm for computing the strong geodetic number of outerplanar graphs. Some general properties about the strong geodetic number were also given. For example, in \cite{wang2020sgnogac} the graphs with the strong geodetic number $2$, $n(G)$, or $n(G)-1$, were charactarized. In \cite{irsic2018sgpocpog} relation between the strong geodetic number of a graph and its induced, convex, or gated subgraphs were derived.

Like geodetic problem and many other problems in graph theory, there is an interesting edge version of the problem. In $2017$, the edge version of the strong geodetic problem, called the strong edge geodetic problem, was introduced in \cite{manuel2017segpin}, but it did not get as much attention as the vertex version. This gap is in part filled in this paper. In the seminal paper \cite{manuel2017segpin} it was proved that the strong edge geodetic problem is NP-complete and some general upper and lower bound were given using the isometric path number, the number of simplicial vertices in a graph, and the number of convex components in graph. In this article, we will show that even though the vertex and edge version of the strong geodetic problem seem similar at the first sign, they differ a lot. For example, in \cite{klavzar2018sgpigla} it was proved that if $2 \leq n \leq m$, then $\sg(P_n \cp P_m) \leq \left\lceil 2 \sqrt{n} ~ \right\rceil$, while we will prove that this is not true for the strong edge geodetic number $\sge(P_n \cp P_m)$ when $m = 3$ or $4$. The main results of this article are the following three theorems.

\begin{theorem}
	\label{theoremsgePnP2}
	If $n \geq 2$, then $\sge(P_n \cp P_2) = \left\lceil 2 \sqrt{n} ~ \right\rceil$.
\end{theorem}

\begin{theorem}
	\label{theoremsgePnP3}
	If $n \geq 2$, then $\sge(P_n \cp P_3) = \left\lceil 2 \sqrt{n+1} ~ \right\rceil$.
\end{theorem}

\begin{theorem}
	\label{theoremsgePnP4}
	If $n \geq 2$, then
	$$\sge(P_n \cp P_4) =
	\begin{cases}
	2k+1; & n=k^2+h, 0 \leq h \leq k-1, \\
	2k+2; & n=k^2+h, k \leq h \leq 2k-1, \\
	2k+3; & n=k^2+2k.
	\end{cases}
	$$
\end{theorem}

Theorem~\ref{theoremsgePnP4} implies that $\sge(P_n \cp P_4) = \left\lceil 2 \sqrt{n+2} ~ \right\rceil$ for all $n \in \N$ except when $n=k^2+k-1$ for some $k \in \N$. This can also be interpreted as $\sge(P_n \cp P_4) = \left\lceil 2 \sqrt{n+1} ~ \right\rceil$ for all $n \in \N$ except when $n=k^2+2k$ for some $k \in \N$. This shows that the pattern from Theorems~\ref{theoremsgePnP2} and \ref{theoremsgePnP3} does not extend to $n \geq 4$.

In the next section we formally define concepts needed in this paper and prepare several preliminary results. Then, in Sections~\ref{sec:proofoftheoremPnP2}-\ref{sec:proofoftheoremPnP4}, we prove Theorems~\ref{theoremsgePnP2}, \ref{theoremsgePnP3}, \ref{theoremsgePnP4}, respectively. In the last section we give three general upper bounds on $\sge(P_n \cp P_m)$.

\section{Preliminaries}
\label{sec:preliminaries}

Let $G = (V(G),E(G))$ be a simple graph. A $x,y$-geodesic is a shortest path between vertices $x$ and $y$. With $P(G;x,y)$ we denote the set of all shortest paths in $G$ between vertices $x$ and $y$. A set $S \subseteq V(G)$ is a \emph{strong edge geodetic set} if there exists an assignment of shortest paths $P_{x,y} \in P(G;x,y)$ for every pair $x,y \in V(G)$, such that $$\bigcup\limits_{\{x,y\} \in \binom{S}{2}} E(P_{x,y}) = E(G),$$ where $E(P_{x,y})$ denotes the set of edges from the selected shortest path $P_{x,y}$. The set of these shortest paths is called the \emph{strong edge covering}. The \emph{strong edge geodetic number} of $G$, denoted by $\sge(G)$, is the cardinality of a smallest strong edge geodetic set of $G$.

\emph{The Cartesian product} $G \cp H$ of graphs $G$ and $H$ is the graph on the vertex set $V(G \cp H) = V(G) \times V(H)$, where two vertices $(g_1,h_1)$ and $(g_2,h_2)$, $g_1,g_2 \in V(G)$, $h_1,h_2 \in V(H)$ are adjecent if $g_1g_2 \in E(G)$ and $h_1 = h_2$, or if $g_1 = g_2$ and $h_1h_2 \in E(H)$. An edge $(g_1,h_1)(g_2,h_2) \in E(G \cp H)$ is said to be \emph{horizontal} if $h_1 = h_2$ and is said to be \emph{vertical} if $g_1 = g_2$.
A \emph{grid} is the Cartesian product of two paths. The \emph{$i$-th row}, $1 \leq i \leq m$, in $P_n \cp P_m$ is the vertex set $\{(1,i), \dots, (n,i)\}$ together with the horizontal edges between them. Similarly, the \emph{$j$-th column}, $1 \leq j \leq n$, in $P_n \cp P_m$ is the vertex set $\{(j,1), \dots, (j,m)\}$ together with the vertical edges between them.    
Because of the commutativity of the Cartesian product operation, $\sge(G \cp H)= \sge(H \cp G)$.
A subgraph $H$ of a graph $G$ is \emph{convex} if for every pair of vertices $\{x,y\}$ in $H$, every $x,y$-geodesic lies completely in $H$. A set of edges $F \subseteq E(G)$ is a \emph{convex edge-cut} if $G - F$ has precisely two convex components.

\begin{lemma}
	\label{lemmaeasylowerboundofPnPm}
	If $n \geq 2$ and $m \geq 2$, then $\sge(P_n \cp P_m) \geq \left\lceil 2 \sqrt{n} ~ \right\rceil$.
\end{lemma}

\begin{proof}
	In \cite[Corollary 6.4]{manuel2017segpin} it is proved that if $F$ is a convex edge-cut of a graph $G$, then $\sge(G) \geq \lceil 2 \sqrt{|F|} \rceil.$ In our case, all vertical edges between the first and the second row in $P_n \cp P_m$ represent a convex edge-cut set of $P_n \cp P_m$ (see Fig.~\ref{figure: convexedgecutinPnPm}). Because there are exactly $n$ vertical edges between the first and the second row in $P_n \cp P_2$, the inequality holds.
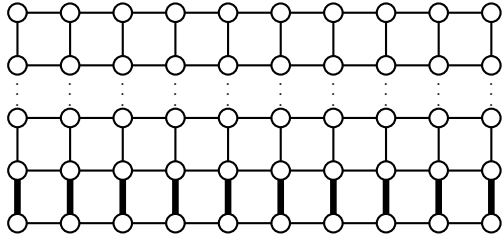
\begin{figure}[ht!]
		\centering
		\begin{tikzpicture}[thick,scale=0.7]
		\foreach \i in {1,...,10}{
			\draw[line width = 2.5pt] (\i,0) -- (\i,1);
		};
		\mreza{10}{0}
		\draw (1,0) -- (10,0); 
		\foreach \i in {1,...,10}{
			\draw[fill=white] (\i,0) circle(5pt);
		}
		\end{tikzpicture}
		\caption{Edge-cut in $P_n \cp P_m$.}
		\label{figure: convexedgecutinPnPm}
	\end{figure}
\end{proof}

Throughout the rest of this paper, we will use Algorithm~\ref{algorithm1verticaledgesinPnPm} to prove upper bounds on the strong edge geodetic number. The input of this algorithm are an integer $n$ and an integer $m$. Integer $n$ can be uniquely  written as the sum $n=k^2+h$, where $k$ and $h$ are integers and $0 \leq h \leq 2k$. Algorithm defines a set of vertices $S$ and a set of shortest paths, where for each pair of vertices in $S$ it uses at most one shortest path between them, such that the union of this shortest paths covers all the vertical edges in $P_n \cp P_m$.

Algorithm~\ref{algorithm1verticaledgesinPnPm} first takes two vertices, $a_1 = (1,1)$ and $b_1 = (1,m)$ from the first column, and cover the vertical edges in the first column by the unique $a_1,b_1$-geodesic. In next step, if $k^2 \geq 4$, the algorithm takes vertices $a_2 = (2^2,1)$ and $b_2 = (2^2,m)$ and covers the second column of edges by the $a_1$,$b_2$-geodesic, the third column of edges by the $a_2,b_1$-geodesic and the fourth column of edges by the $a_2,b_2$-geodesic. After the $(i-1)$-th, $(i-1)^2 \leq n$, step, all the first $(i-1)^2$ columns are already covered. In the $i$-th step, if $i^2 \leq n$, we add vertices $a_i = (i^2,1)$ and $b_i=(i^2,m)$. We cover the next $i-1$ columns by $a_1,b_i$-, \dots, $a_{i-1},b_i$-geodesic, respectively the way every geodesics covers the leftmost not yet covered column of edges. Similarly, we cover $i-1$ edges from the $(i-1)^2+(i-1)+1$-th to the $(i-1)^2+2(i-1)$-th column by $a_i,b_1$-, \dots, $a_i,b_{i-1}$-geodesic, respectively, covering the leftmost not yet covered column of edges. The algorithm covers the $i$-th column of edges by the unique $a_i,b_i$-geodesic. This way we cover all the vertical edges from the first $k^2$ columns.

If $1 \leq h \leq k$, we add the vertex $b_{k+1}=(n,m)$ and then cover the remaining columns of edges by $a_1,b_{k+1}$-, \dots, $a_h,b_{k+1}$-geodesic, respectively, covering the leftmost not yet covered column of edges.
If $k+1 \leq h \leq 2k$, we add vertices $a_{k+1}=(n,1)$ and $b_{k+1}=(n,m)$. We cover the next $k$ columns by $a_1,b_{k+1}$-, \dots, $a_k,b_{k+1}$-geodesic, respectively, the way every geodesics covers the leftmost not yet covered column of edges. Simirarly, we cover the remaining columns of edges by $a_{k+1},b_1$-, \dots, $a_{k+1},b_{h-k}$-geodesic, respectively, covering the leftmost not yet covered column of edges. 

\begin{algorithm}[ht!]
\SetAlgoLined
\vspace{0.2cm}
\KwInput{integer $n=k^2+h$, where $0 \leq h \leq 2k$, and integer $m$}
\vspace{0.2cm}
\KwResult{strong edge geodetic set and a set of shortest paths in $P_n \cp P_m$, $m \geq 2$, that cover all vertical edges}
\vspace{0.2cm}
 \For{$i = 1, \dots, k$}{
 	$a_i = \left(i^2,1\right)$\;
 	$b_i = \left(i^2,m\right)$\;
 	\For{$j = 1, \dots, i-1$}{
 		\parbox[t]{313pt}{connect vertices $a_i$ and $b_j$ with a geodesic that covers all vertical edges in $((i-1)^2+j)$-th column of $P_n \cp P_m$ \strut}
 	}
 	\For{$j = 1, \dots, i-1$}{
		\parbox[t]{313pt}{connect vertices $a_j$ and $b_i$ with a geodesic that covers all vertical edges in $(i(i-1)+j)$-th column of $P_n \cp P_m$ \strut}
	}
	connect vertices $a_i$ and $b_i$ with the unique geodesic between them
 }
 \If{$1 \leq h \leq k$}{
 	$b_{k+1} = \left(n,m\right)$\;
 	\For{$j = 1, \dots, h$}{
 		\parbox[t]{313pt}{connect vertices $a_j$ and $b_{k+1}$ with a geodesic that covers all vertical edges in $(k^2+j)$-th column of $P_n \cp P_m$ \strut}
 	}
 }
 \ElseIf{$k+1 \leq h \leq 2k$}{
 	$a_{k+1} = \left(n,1\right)$\;
 	$b_{k+1} = \left(n,m\right)$\;
 	\For{$j = 1, \dots, k$}{
 		\parbox[t]{313pt}{connect vertices $a_j$ and $b_{k+1}$ with a geodesic that covers all vertical edges in $(k^2+j)$-th column of $P_n \cp P_m$ \strut}
 	}
	\For{$j = 1, \dots, h-k$}{
		\parbox[t]{313pt}{connect vertices $a_{k+1}$ and $b_j$ with a geodesic that covers all vertical edges in $(k(k+1)+j)$-th column of $P_n \cp P_m$ \strut}
	}
 }
 \vspace{0.2cm}
 \caption{Covering vertical edges in grids}
 \label{algorithm1verticaledgesinPnPm}
\end{algorithm}

We conclude the preliminaries with the following technical lemma to be used in our proofs.
\begin{lemma}
	\label{numberceil2timessqrtn}
	$$\left\lceil 2 \sqrt{n\phantom{!}} \right\rceil = 
	\begin{cases}
	2k; & n = k^2, ~ k \in \N, \\
	2k+1; & n = k^2 + h, ~ k \in \N, 1 \leq h \leq k, \\
	2k+2; & n = k^2 + h, ~ k \in \N, k+1 \leq h \leq 2k.
	\end{cases}
	$$
\end{lemma}

\begin{proof}
	If $n = k^2$ for some $k \in \N$, then $\left\lceil 2 \sqrt{n\phantom{!}} \right\rceil = \left\lceil 2 \sqrt{k^2} \right\rceil = 2k$.
	
	Suppose $n = k^2 + h$ for some $k,h \in \N$ where $1 \leq h \leq k$. Then because $h>0$ and $\left\lceil 2 \sqrt{k^2} \right\rceil$ is an integer, it holds 
	$$\left\lceil 2 \sqrt{n\phantom{!}} \right\rceil = \left\lceil 2 \sqrt{k^2+h} \right\rceil > \left\lceil 2 \sqrt{k^2} \right\rceil = 2k.$$
	Also, because $h \leq k$ and $\left\lceil 2 \sqrt{k^2} \right\rceil$ is an integer, it holds $$\left\lceil 2 \sqrt{n\phantom{!}} \right\rceil \leq \left\lceil 2 \sqrt{k^2+k} \right\rceil = \left\lceil 2 \sqrt{(k^2+1/2)^2-1/4} \right\rceil \leq \left\lceil 2 \sqrt{(k+1/2)^2} \right\rceil = 2k+1.$$
	Because $\left\lceil 2 \sqrt{n} ~ \right\rceil$ is an integer greater than $2k$ and is also less or equal to $2k+1$, we conclude that $\left\lceil 2 \sqrt{n} ~ \right\rceil = 2k+1$ when $n=k^2+h$, $1 \leq h \leq k$.
	
	Now suppose that $n = k^2 + h$ for some $k,h \in \N$ where $k+1 \leq h \leq 2k$.
	Then because $h \geq k+1$ and $\left\lceil 2 \sqrt{(k+1/2)^2} \right\rceil$ is an integer, it holds
	$$\left\lceil 2 \sqrt{n\phantom{!}} \right\rceil \geq \left\lceil 2 \sqrt{k^2+k+1} \right\rceil = \left\lceil 2 \sqrt{(k+1/2)^2+3/4} \right\rceil > \left\lceil 2 \sqrt{(k+1/2)^2} \right\rceil = 2k+1.$$
	Also, because $h \leq 2k$ and $\left\lceil 2 \sqrt{k^2} \right\rceil$ is an integer, it holds $$\left\lceil 2 \sqrt{n\phantom{!}} \right\rceil \leq \left\lceil 2 \sqrt{k^2+2k} \right\rceil \leq \left\lceil 2 \sqrt{k^2+2k+1} \right\rceil = 2k+2.$$
	Because $\left\lceil 2 \sqrt{n} ~ \right\rceil$ is an integer greater than $2k+1$ and is also less or equal to $2k+2$, we get $\left\lceil 2 \sqrt{n} ~ \right\rceil = 2k+2$ when $n=k^2+h$, $k+1 \leq h \leq 2k$.
\end{proof}

\section{Proof of Theorem~\ref{theoremsgePnP2}}
\label{sec:proofoftheoremPnP2}

The lower bound $\sge(P_n \cp P_2) \geq \left\lceil 2 \sqrt{n} ~ \right\rceil$ follows from Lemma~\ref{lemmaeasylowerboundofPnPm}. It remains to prove the upper bound $\sge(P_n \cp P_2) \leq \left\lceil 2 \sqrt{n} ~ \right\rceil$.

We will use Algorithm~\ref{algorithm1verticaledgesinPnPm} for $m=2$ to cover all the vertical edges in $P_n \cp P_2$. When $n = k^2$ for some $k \in \N$, the number of vertices used in Algorithm~\ref{algorithm1verticaledgesinPnPm} is exactly $2k$, these are the vertices $a_1, \dots, a_k, b_1, \dots, b_k$. When $n=k^2+h$, $k \in \N$, $1 \leq h \leq k$, the number of vertices used in Algorithm~\ref{algorithm1verticaledgesinPnPm} is exactly $2k+1$, and when $n=k^2+h$, $k \in \N$, $k+1 \leq h \leq 2k$, exactly $2k+1$ vertices are used. By Lemma~\ref{numberceil2timessqrtn}, we see that Algorithm~\ref{algorithm1verticaledgesinPnPm} uses exactly $\lceil 2 \sqrt{n} ~ \rceil$ vertices.
	
	It remains to cover the horizontal edges. Observe that Algorithm~\ref{algorithm1verticaledgesinPnPm}
	 uses only the $a_{j_1}$,$b_{j_2}$-geodesics for some $j_1,j_2 \in \N$.
	 
	 If $n = k^2$, all horizontal edges from the first row can be covered by the unique $a_1$,$a_k$-geodesic. Similarly, all horizontal edges from the second row can be covered by the unique $b_1$,$b_k$-geodesic.
	 
	 If $n = k^2 + h$, where $1 \leq h \leq k$, we can cover the horizontal edges in the first row with the unique $a_1$,$a_k$-geodesic. The remaining horizontal edges in the first row are already covered with the shortest path from Algorithm~\ref{algorithm1verticaledgesinPnPm} that covers vertical edges in the $n$-th column. We can cover the horizontal edges from the second row with the unique $b_1$,$b_{k+1}$-geodesic. 
	 
	 If $n=k^2+h$, where $k+1 \leq h \leq 2k$, we can cover the horizontal edges from the first row with the unique $a_1$,$a_{k+1}$-geodesic, and the horizontal edges from the second row with the unique $b_1$,$b_{k+1}$-geodesic.
	 
	 This proves Theorem~\ref{theoremsgePnP2}. See Fig.~\ref{figure:typicaloptimalexamplesstrongedgegeodeticset} for some typical optimal strong edge geodetic sets in $P_n \cp P_2$.

\begin{figure}[ht!]
	\centering
	\begin{tikzpicture}[thick,scale=0.5]
	\mreza{16}{2}
	\draw[fill=black] (1,1) circle(5pt);
	\draw[fill=black] (1,2) circle(5pt);
	\draw[fill=black] (4,1) circle(5pt);
	\draw[fill=black] (4,2) circle(5pt);
	\draw[fill=black] (9,1) circle(5pt);
	\draw[fill=black] (9,2) circle(5pt);
	\draw[fill=black] (16,1) circle(5pt);
	\draw[fill=black] (16,2) circle(5pt);
	\end{tikzpicture}
	
	\vspace{0.5cm}
	
	\begin{tikzpicture}[thick,scale=0.5]
	\mreza{20}{2}
	\draw[fill=black] (1,1) circle(5pt);
	\draw[fill=black] (1,2) circle(5pt);
	\draw[fill=black] (4,1) circle(5pt);
	\draw[fill=black] (4,2) circle(5pt);
	\draw[fill=black] (9,1) circle(5pt);
	\draw[fill=black] (9,2) circle(5pt);
	\draw[fill=black] (16,1) circle(5pt);
	\draw[fill=black] (16,2) circle(5pt);
	\draw[fill=black] (20,2) circle(5pt);
	\end{tikzpicture}
	
	\vspace{0.5cm}
	
	\begin{tikzpicture}[thick,scale=0.5]
	\mreza{24}{2}
	\draw[fill=black] (1,1) circle(5pt);
	\draw[fill=black] (1,2) circle(5pt);
	\draw[fill=black] (4,1) circle(5pt);
	\draw[fill=black] (4,2) circle(5pt);
	\draw[fill=black] (9,1) circle(5pt);
	\draw[fill=black] (9,2) circle(5pt);
	\draw[fill=black] (16,1) circle(5pt);
	\draw[fill=black] (16,2) circle(5pt);
	\draw[fill=black] (24,1) circle(5pt);
	\draw[fill=black] (24,2) circle(5pt);
	\end{tikzpicture}
	\caption{Strong edge geodetic set for graphs $P_{16} \cp P_2$, $P_{20} \cp P_2$ and $P_{24} \cp P_2$.}
	\label{figure:typicaloptimalexamplesstrongedgegeodeticset}
\end{figure}
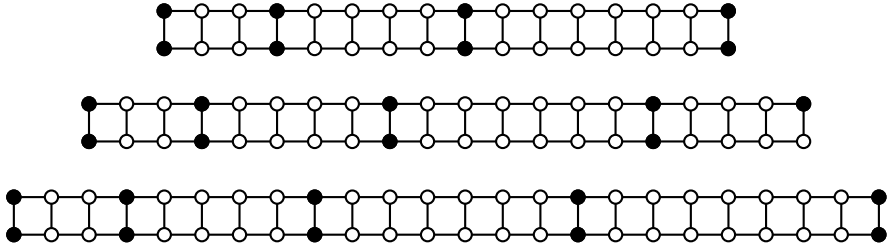

\section{Proof of Theorem~\ref{theoremsgePnP3}}
\label{sec:proofoftheoremPnP3}
We first show:

\begin{lemma}
	\label{lemmaeasyupperboundsgePnP3}
	If $n \geq 2$, then $\sge(P_n \cp P_3) \leq \left\lceil 2 \sqrt{n} ~ \right\rceil +1$.
\end{lemma}

\begin{proof}
	Use Algorithm~\ref{algorithm1verticaledgesinPnPm} and then cover horizontal edges in the first row and in the last row in the same way as in the proof of Theorem~\ref{theoremsgePnP2}. Add vertex $(n,2)$ to the existing vertex set from Algorithm~\ref{algorithm1verticaledgesinPnPm}, and connect it with $(1,1)$ by a geodesic that covers all horizontal edges from the second row, see Fig.~\ref{figure: covering the middle row in Pn cp P3}.
\end{proof}
	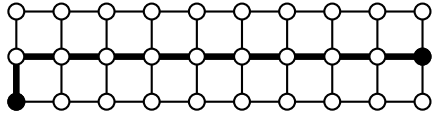
\begin{figure}[ht!]
		\centering
		\begin{tikzpicture}[thick,scale=0.6]
		\draw[line width = 2.5pt] (1,1) -- (1,2) -- (10,2);
		\mreza{10}{3}
		\draw[fill=black] (1,1) circle(5pt);
		\draw[fill=black] (10,2) circle(5pt);
		\end{tikzpicture}
		\caption{Shortest path between $(1,1)$ and $(n,2)$.}
		\label{figure: covering the middle row in Pn cp P3}
	\end{figure}

	By Lemma~\ref{lemmaeasyupperboundsgePnP3} and Lemma~\ref{lemmaeasylowerboundofPnPm} (for $m = 3$) we know that for every $n \geq 2$, the strong edge goedetic number $\sge(P_n \cp P_3)$ is either $\left\lceil 2 \sqrt{n} ~ \right\rceil$ or $\left\lceil 2 \sqrt{n} ~ \right\rceil + 1$. Because by Lemma~\ref{numberceil2timessqrtn} it holds $\left\lceil 2 \sqrt{n} ~ \right\rceil = \left\lceil 2 \sqrt{n+1} ~ \right\rceil$, except when $n=k^2$ or $n=k^2+k$ for some $k \in \N$, it is enough to prove that when $n = k^2$ or $n = k^2+k$, there is no strong edge geodetic set of size $\left\lceil 2 \sqrt{n} ~ \right\rceil$ and to find a strong edge geodetic set of size $\left\lceil 2 \sqrt{n} ~ \right\rceil$ in the other cases.
	
	First, let us show that there exists a strong edge geodetic set of size $\left\lceil 2 \sqrt{n} ~ \right\rceil$ on $P_n \cp P_3$ for $n = k^2 + h$, where $1 \leq h \leq k-1$ or $k+1 \leq h \leq 2k$.
	
	\begin{proposition}
		\label{propositionsgePnP3moreexamples}
		If $n=k^2+h$, $n \geq 3$, $k \geq 1$, where $1 \leq h \leq k-1$ or $k+1 \leq h \leq 2k$, then 
		$$\sge(P_n \cp P_3) \leq \sge(P_n \cp P_2).$$
	\end{proposition}
	
	\begin{proof}
	To cover all the vertical edges and all the horizontal edges in the second row, we will adjust Algorithm~\ref{algorithm1verticaledgesinPnPm}. We will divide this adjustment into two cases.
	
	\begin{case} $n = k^2 + h$, $1 \leq h \leq k-1$ (see Fig.~\ref{figure: 1proofforupperboundinPncpP3}). \\
		We can see that in this case, because $h \leq k-1$, Algorithm~\ref{algorithm1verticaledgesinPnPm} never uses a $a_k$,$b_{k+1}$-geodesic and that a $a_1$,$b_{k+1}$-geodesic covers all vertical edges in the $(k^2+1)$-th column. Because $a_k = (k^2,1)$, we can add a $a_k$,$b_{k+1}$-geodesic that covers all the vertical edges in the $(k^2+1)$-th column and then replace the existing $a_1$,$b_{k+1}$-geodesic with the one that covers all the horizontal edges in the second row. This way we have covered all the vertical edges in $P_n \cp P_3$ and also all the horizontal edges in the second row.
		
		In the same way as in $P_n \cp P_2$ we can cover the horizontal edges in the first row (using the $a_1$,$a_k$-geodesic and also the existing $a_h$,$b_{k+1}$-geodesic) and in the third row (using the $b_1$,$b_{k+1}$-geodesic) of $P_n \cp P_3$.
		
		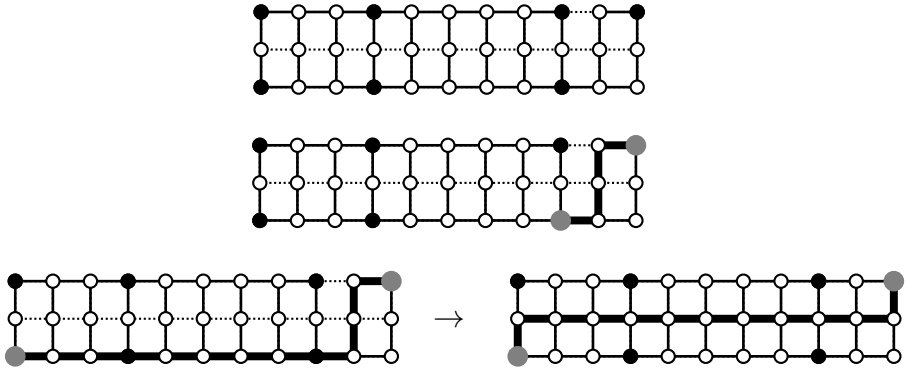
\begin{figure}[ht!]
			\centering
			\begin{tikzpicture}[thick,scale=0.5]
			\foreach \i in {1,...,10}{
				\draw[line width = 1pt] (\i,1) -- (\i,3);
			};
			\draw[line width = 1pt] (11,1) -- (11,3);
			\draw[line width = 1pt] (1,1) -- (11,1);
			\draw[line width = 1pt] (1,3) -- (9,3);
			\draw[line width = 1pt] (10,3) -- (11,3);
			\mrezacrtkano{11}{3}
			\draw[fill=black] (1,1) circle(5pt);
			\draw[fill=black] (1,3) circle(5pt);
			\draw[fill=black] (4,1) circle(5pt);
			\draw[fill=black] (4,3) circle(5pt);
			\draw[fill=black] (9,1) circle(5pt);
			\draw[fill=black] (9,3) circle(5pt);
			\draw[fill=black] (11,3) circle(5pt);
			\end{tikzpicture}
			\\
			\vspace{0.5cm}
			\begin{tikzpicture}[thick,scale=0.5]
			\foreach \i in {1,...,11}{
				\draw[line width = 1pt] (\i,1) -- (\i,3);
			};
			\draw[line width = 1pt] (1,1) -- (11,1);
			\draw[line width = 1pt] (1,3) -- (9,3);
			\draw[line width = 1pt] (10,3) -- (11,3);
			\draw[line width = 3pt] (9,1) -- (10,1) -- (10,3) -- (11,3);
			\mrezacrtkano{11}{3}
			\draw[fill=black] (1,1) circle(5pt);
			\draw[fill=black] (1,3) circle(5pt);
			\draw[fill=black] (4,1) circle(5pt);
			\draw[fill=black] (4,3) circle(5pt);
			\draw[fill=gray, color=gray] (9,1) circle(7pt);
			\draw[fill=black] (9,3) circle(5pt);
			\draw[fill=gray, color=gray] (11,3) circle(7pt);
			\end{tikzpicture}
			\\
			\vspace{0.5cm}
			\begin{tabular}{c c c}
				\adjustbox{valign=c}{
					\begin{tikzpicture}[thick,scale=0.5]
					\foreach \i in {1,...,11}{
						\draw[line width = 1pt] (\i,1) -- (\i,3);
					};
					\draw[line width = 1pt] (1,1) -- (11,1);
					\draw[line width = 1pt] (1,3) -- (9,3);
					\draw[line width = 1pt] (10,3) -- (11,3);
					\draw[line width = 3pt] (1,1) -- (10,1) --(10,3) -- (11,3);
					\mrezacrtkano{11}{3}
					\draw[fill=gray, color=gray] (1,1) circle(7pt);
					\draw[fill=black] (1,3) circle(5pt);
					\draw[fill=black] (4,1) circle(5pt);
					\draw[fill=black] (4,3) circle(5pt);
					\draw[fill=black] (9,1) circle(5pt);
					\draw[fill=black] (9,3) circle(5pt);
					\draw[fill=gray, color=gray] (11,3) circle(7pt);
					\end{tikzpicture}}
				& $\to$ &
				\adjustbox{valign=c}{
					\begin{tikzpicture}[thick,scale=0.5]
					\foreach \i in {1,...,10}{
						\draw[line width = 1pt] (\i,1) -- (\i,3);
					};
					\draw[line width = 1pt] (11,1) -- (11,3);
					\draw[line width = 1pt] (1,1) -- (11,1);
					\draw[line width = 1pt] (1,3) -- (11,3);
					\draw[line width = 3pt] (1,1) -- (1,2) -- (11,2) -- (11,3);
					\mrezacrtkano{11}{3}
					\draw[fill=gray, color=gray] (1,1) circle(7pt);
					\draw[fill=black] (1,3) circle(5pt);
					\draw[fill=black] (4,1) circle(5pt);
					\draw[fill=black] (4,3) circle(5pt);
					\draw[fill=black] (9,1) circle(5pt);
					\draw[fill=black] (9,3) circle(5pt);
					\draw[fill=gray, color=gray] (11,3) circle(7pt);
					\end{tikzpicture}}
			\end{tabular}
			\caption{Strong edge geodetic set in $P_{3^2+2} \cp P_3$.}
			\label{figure: 1proofforupperboundinPncpP3}
		\end{figure}
	\end{case}
	
	\begin{case} $n = k^2 + h$, $k+1 \leq h \leq 2k$ (see Fig.~\ref{figure: 2proofforupperboundinPncpP3}). \\
		In this case Algorithm~\ref{algorithm1verticaledgesinPnPm} does not use the $a_{k+1}$,$b_{k+1}$-geodesic. This shortest path covers all the vertical edges in the $n$-th column. Because a $a_{k+1}$,$b_h$-geodesic from Algorithm~\ref{algorithm1verticaledgesinPnPm} also covers all the vertical edges from the $n$-th column, we can add a $a_{k+1}$,$b_{k+1}$-geodesic and then replace an existing $a_{k+1}$,$b_h$-geodesic with the one that covers all the vertical edges in the $(k^2+1)$-th column. This way we can then similar to the previous case replace the existing $a_1$,$b_{k+1}$-geodesic (which covers all the vertical edges in the $(k^2+1)$-th column in Algorithm~\ref{algorithm1verticaledgesinPnPm}) with the one that covers all the horizontal edges in the second row of $P_n \cp P_3$.
		
		In the same way as in $P_n \cp P_2$ we now cover the horizontal edges in the first row (using the $a_1$,$a_{k+1}$-geodesic) and in the third row (using the $b_1$,$b_{k+1}$-geodesic) of $P_n \cp P_3$.
		
		\begin{figure}[ht!]
			\centering
			\begin{tikzpicture}[thick,scale=0.46]
			\foreach \i in {1,...,15}{
				\draw[line width = 1pt] (\i,1) -- (\i,3);
			};
			\draw[line width = 1pt] (1,1) -- (12,1);
			\draw[line width = 1pt] (13,1) -- (15,1);
			\draw[line width = 1pt] (1,3) -- (15,3);
			\mrezacrtkano{15}{3}
			\draw[fill=black] (1,1) circle(5pt);
			\draw[fill=black] (1,3) circle(5pt);
			\draw[fill=black] (4,1) circle(5pt);
			\draw[fill=black] (4,3) circle(5pt);
			\draw[fill=black] (9,1) circle(5pt);
			\draw[fill=black] (9,3) circle(5pt);
			\draw[fill=black] (15,1) circle(5pt);
			\draw[fill=black] (15,3) circle(5pt);
			\end{tikzpicture}
			\\
			\vspace{0.5cm}
			\begin{tikzpicture}[thick,scale=0.46]
			\foreach \i in {1,...,15}{
				\draw[line width = 1pt] (\i,1) -- (\i,3);
			};
			\draw[line width = 1pt] (1,1) -- (12,1);
			\draw[line width = 1pt] (13,1) -- (15,1);
			\draw[line width = 1pt] (1,3) -- (15,3);
			\draw[line width = 3pt] (15,1) -- (15,3);
			\mrezacrtkano{15}{3}
			\draw[fill=black] (1,1) circle(5pt);
			\draw[fill=black] (1,3) circle(5pt);
			\draw[fill=black] (4,1) circle(5pt);
			\draw[fill=black] (4,3) circle(5pt);
			\draw[fill=black] (9,1) circle(5pt);
			\draw[fill=black] (9,3) circle(5pt);
			\draw[fill=gray, color=gray] (15,1) circle(7pt);
			\draw[fill=gray, color=gray] (15,3) circle(7pt);
			\end{tikzpicture}
			\\
			\vspace{0.5cm}
			\begin{tabular}{c c c}
				\adjustbox{valign=c}{
					\begin{tikzpicture}[thick,scale=0.46]
					\foreach \i in {1,...,15}{
						\draw[line width = 1pt] (\i,1) -- (\i,3);
					};
					\draw[line width = 1pt] (1,1) -- (12,1);
					\draw[line width = 1pt] (13,1) -- (15,1);
					\draw[line width = 1pt] (1,3) -- (15,3);
					\draw[line width = 3pt] (9,3) -- (15,3) -- (15,1);
					\mrezacrtkano{15}{3}
					\draw[fill=black] (1,1) circle(5pt);
					\draw[fill=black] (1,3) circle(5pt);
					\draw[fill=black] (4,1) circle(5pt);
					\draw[fill=black] (4,3) circle(5pt);
					\draw[fill=black] (9,1) circle(5pt);
					\draw[fill=gray, color=gray] (9,3) circle(7pt);
					\draw[fill=gray, color=gray] (15,1) circle(7pt);
					\draw[fill=black] (15,3) circle(5pt);
					\end{tikzpicture}}
				& $\to$ &
				\adjustbox{valign=c}{
					\begin{tikzpicture}[thick,scale=0.46]
					\foreach \i in {1,...,15}{
						\draw[line width = 1pt] (\i,1) -- (\i,3);
					};
					\draw[line width = 1pt] (1,1) -- (12,1);
					\draw[line width = 1pt] (13,1) -- (15,1);
					\draw[line width = 1pt] (1,3) -- (15,3);
					\draw[line width = 3pt] (9,3) -- (10,3) -- (10,1) -- (15,1);
					\mrezacrtkano{15}{3}
					\draw[fill=black] (1,1) circle(5pt);
					\draw[fill=black] (1,3) circle(5pt);
					\draw[fill=black] (4,1) circle(5pt);
					\draw[fill=black] (4,3) circle(5pt);
					\draw[fill=black] (9,1) circle(5pt);
					\draw[fill=gray, color=gray] (9,3) circle(7pt);
					\draw[fill=gray, color=gray] (15,1) circle(7pt);
					\draw[fill=black] (15,3) circle(5pt);
					\end{tikzpicture}}
				\\
			\end{tabular}
			\\
			\vspace{0.5cm}
			\begin{tabular}{c c c}
				\adjustbox{valign=c}{
					\begin{tikzpicture}[thick,scale=0.46]
					\foreach \i in {1,...,15}{
						\draw[line width = 1pt] (\i,1) -- (\i,3);
					};
					\draw[line width = 1pt] (1,1) -- (15,1);
					\draw[line width = 1pt] (1,3) -- (14,3);
					\draw[line width = 3pt] (1,1) -- (10,1) -- (10,3) -- (15,3);
					\mrezacrtkano{15}{3}
					\draw[fill=gray, color=gray] (1,1) circle(7pt);
					\draw[fill=black] (1,3) circle(5pt);
					\draw[fill=black] (4,1) circle(5pt);
					\draw[fill=black] (4,3) circle(5pt);
					\draw[fill=black] (9,1) circle(5pt);
					\draw[fill=black] (9,3) circle(5pt);
					\draw[fill=gray, color=gray] (15,3) circle(7pt);
					\draw[fill=black] (15,1) circle(5pt);
					\end{tikzpicture}}
				& $\to$ &
				\adjustbox{valign=c}{
					\begin{tikzpicture}[thick,scale=0.46]
					\foreach \i in {1,...,15}{
						\draw[line width = 1pt] (\i,1) -- (\i,3);
					};
					\draw[line width = 1pt] (1,1) -- (15,1);
					\draw[line width = 1pt] (1,3) -- (15,3);
					\draw[line width = 3pt,] (1,1) -- (1,2) -- (15,2) -- (15,3);
					\mrezacrtkano{15}{3}
					\draw[fill=gray, color=gray] (1,1) circle(7pt);
					\draw[fill=black] (1,3) circle(5pt);
					\draw[fill=black] (4,1) circle(5pt);
					\draw[fill=black] (4,3) circle(5pt);
					\draw[fill=black] (9,1) circle(5pt);
					\draw[fill=black] (9,3) circle(5pt);
					\draw[fill=gray, color=gray] (15,3) circle(7pt);
					\draw[fill=black] (15,1) circle(5pt);
					\end{tikzpicture}}
			\end{tabular}
			\caption{Strong edge geodetic set in $P_{3^2+2 \cdot 3} \cp P_3$.}
			\label{figure: 2proofforupperboundinPncpP3}
		\end{figure}
	\end{case}
	
	In both cases we have adjusted Algorithm~\ref{algorithm1verticaledgesinPnPm} such that all the vertical edges together with all the horizontal edges are covered without changing the set of vertices used in the algorithm. It follows that $\sge(P_n \cp P_3) \leq \left\lceil 2 \sqrt{n} ~ \right\rceil = \left\lceil 2 \sqrt{n+1} ~ \right\rceil$ for all $n \in \N$ when $n \not= k^2$ or $n \not= k^2+k$ for some $k \in \N$.
	\end{proof}

	In the second part of the proof we will prove that there is no strong edge geodetic set of size $\left\lceil 2 \sqrt{n} ~ \right\rceil$ if $n = k^2$ or $n = k^2 + k$.
	
	\begin{proposition}
		\label{propositionsgePnP3specialexamples}
		If $n = k^2$ or $n = k^2 + k$ for some $k \geq 2$, $k \in \N$, then
		$$\sge(P_n \cp P_3) > \sge(P_n \cp P_2).$$
	\end{proposition}

	\begin{proof}
	Let us define the function
	$$f_s (a,b,c) = a b + b c + 2ac$$
	which gives an upper bound on how many different vertical edges can be covered with shortest paths between $s$ vertices, where $a$ vertices lie in the first row, $b$ vertices in the second row, and $c$ vertices in the third row in $P_n \cp P_3$, and for each pair of these vertices we use at most one shortest path. The extremes of $f_s$ can be obtained by computer. For example, if $a \geq 0$, $b \geq 0$, and $c \geq 0$, $a,b,c \in \R$, then the maximum of $f_s$ is equal to $s^2/2$ and if $a \geq 0$, $b \geq 1$, and $c \geq 0$, $a,b,c \in \R$, then the maximum of $f_s$ is equal to $(s^2-1)/2$. We prove Proposition~\ref{propositionsgePnP3specialexamples} by assuming the opposite in each case.
	
	\setcaseprenumber{1}
	\begin{case} $n = k^2$. \\
		Suppose that there exists a strong edge geodetic set  $S$ in $P_n \cp P_3$ with $2k$ elements. If this set contains a vertex from the second row (in which case $b \geq 1$), $f_s(a,b,c)$ is at most $(s^2-1)/2$, which is in our case (when $s=2k$) equal to $(4k^2-1)/2$. This is less than the number $2k^2$ of all the vertical edges in $P_n \cp P_2$, which in turn means that if there exists such a strong edge geodetic set, then all the vertices from it are in the first and the third row ($b=0$).
		
		Now consider the edge $(1,2)(2,2)$. A shortest path that covers this edge has to have one endpoint at $(1,1)$ or $(1,3)$, which means that at least one of these vertices is in $S$. Without lost of generality we can assume that it has one endpoint at $(1,1)$ (Fig.~\ref{figureforn=k^2inPnP3}a). This shortest path then also covers the edge $(1,1)(1,2)$. If $(1,3) \in S$, then the edge $(1,1)(1,2)$ is also covered with the unique $(1,1)$,$(1,3)$-geodesic (Fig.~\ref{figureforn=k^2inPnP3}b). Otherwise, the shortest path that covers the edge $(1,2)(1,3)$ also covers the edge $(1,1)(1,2)$ (Fig.~\ref{figureforn=k^2inPnP3}c). In both ways the edge $(1,1)(1,2)$ is covered at least twice, which means that with the set of $2k$ vertices we can only find shortest paths between them that cover at most $\max(f_{2k} - 1) \leq 2k^2-1$ different vertical edges which is again less that the number $2k^2$ of all the vertical edges in $P_n \cp P_2$. This implies that such a strong edge geodetic set $S$ cannot exist.
	
		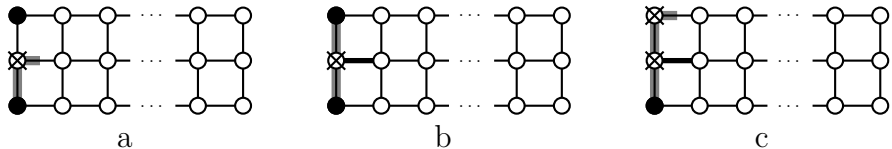
\begin{figure}[ht!]
			\centering	
			\begin{tabular}{c c c c c}
				\adjustbox{valign=c}{
					\begin{tikzpicture}[thick,scale=0.6]
					\draw[line width = 3.5pt, color=gray] (1,1) -- (1,2) -- (1.5,2);
					\mreza{0}{3}
					\draw[line width = 0.8 pt] (0.8,2.2) -- (1.2,1.8);
					\draw[line width = 0.8 pt] (1.2,2.2) -- (0.8,1.8);
					\draw[fill=black] (1,1) circle(5pt);
					\draw[fill=black] (1,3) circle(5pt);
					\end{tikzpicture}}		
				& &
				\adjustbox{valign=c}{
					\begin{tikzpicture}[thick,scale=0.6]
					\draw[line width = 2pt] (1,2) -- (2,2);
					\draw[line width = 3.5pt, color=gray] (1,1) -- (1,3);
					\mreza{0}{3}
					\draw[line width = 0.8 pt] (0.8,2.2) -- (1.2,1.8);
					\draw[line width = 0.8 pt] (1.2,2.2) -- (0.8,1.8);
					\draw[fill=black] (1,1) circle(5pt);
					\draw[fill=black] (1,3) circle(5pt);
					\end{tikzpicture}}
				& &
				\adjustbox{valign=c}{
					\begin{tikzpicture}[thick,scale=0.6]
					\draw[line width = 2pt] (1,2) -- (2,2);
					\draw[line width = 3.5pt, color=gray] (1,1) -- (1,3) -- (1.5,3);
					\mreza{0}{3}
					\draw[line width = 0.8 pt] (0.8,2.2) -- (1.2,1.8);
					\draw[line width = 0.8 pt] (1.2,2.2) -- (0.8,1.8);
					\draw[line width = 0.8 pt] (0.8,3.2) -- (1.2,2.8);
					\draw[line width = 0.8 pt] (1.2,3.2) -- (0.8,2.8);
					\draw[fill=black] (1,1) circle(5pt);
					\end{tikzpicture}}
				\\
				a & & b & & c
			\end{tabular}
			\caption{Covering the edge $(1,2)(2,2)$ in $P_n \cp P_3$ when $n = k^2$.}
			\label{figureforn=k^2inPnP3}
		\end{figure}  
	\end{case}
	
	\begin{case} $n = k^2 + k$. \\
		Suppose that there exist a strong edge geodetic set $S$ of $P_n \cp P_3$ with $2k+1$ elements. If this set does not include a vertex $(1,2)$, we can similarly as in the previous case without loss of generality conclude that the edge $(1,1)(1,2)$ is covered at least twice, which implies that with a set of $2k+1$ vertices, we can only find shortest paths between them that cover at most $\max(f_{2k+1} - 1) \leq (2k+1)^2/2-1 = 2k^2+2k-1/2$ different vertical edges which is less that the number $2(k^2+2)$ of all the vertical edges in $P_n \cp P_3$. This implies that if such a strong edge geodetic set in $P_n \cp P_3$ exists, it includes the vertex $(1,2)$. By symmetry, it also includes the vertex $(n,2)$. But then, because $b \geq 2$, the value of $f_{2k+1}(a,b,c)$ is less or equal to $((2k+1)^2-4)/2$ which is again less than the number $2k^2 + 2k$ of different vertical edges in $P_n \cp P_3$.
	\end{case}
	Since in both cases we got a contradiction, we conclude that $\sge(P_n \cp P_3) > \sge(P_n \cp P_2)$.
	\end{proof}

	Propositions~\ref{propositionsgePnP3moreexamples} and \ref{propositionsgePnP3specialexamples}, together with Lemma~\ref{lemmaeasylowerboundofPnPm} for $m = 3$, and Lemma~\ref{lemmaeasyupperboundsgePnP3} imply Theorem~\ref{theoremsgePnP3}.

\section{Proof of Theorem~\ref{theoremsgePnP4}}
\label{sec:proofoftheoremPnP4}

\begin{lemma}
	\label{lemmaeasyupperboundsgePnP4}
	$\sge(P_n \cp P_4) \leq \left\lceil 2 \sqrt{n} ~ \right\rceil +1$.
\end{lemma}

\begin{proof}
	Use Algorithm~\ref{algorithm1verticaledgesinPnPm} and cover horizontal edges in the first row and in the last row in the same way as in the proof of Theorem~\ref{theoremsgePnP2}. To the existing vertex set from Algorithm~\ref{algorithm1verticaledgesinPnPm}, add vertex $(n,2)$ and connect it with $(1,1)$ by a geodesic that covers all horizontal edges from the second row (Fig.~\ref{figurelemmasgePnP4easyupperbound}a) and with $(1,4)$ by a geodesic that covers all horizontal edges from the third row (Fig.~\ref{figurelemmasgePnP4easyupperbound}b).
	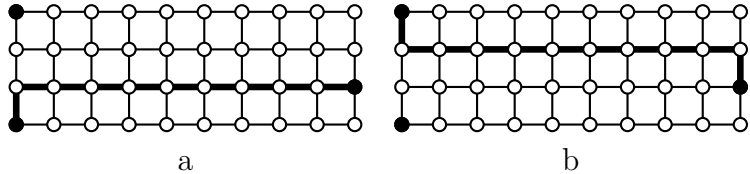
\begin{figure}[ht!]
		\centering
		\begin{tabular}{c c}
			\begin{tikzpicture}[thick,scale=0.5]
			\draw[line width = 2.5pt] (1,1) -- (1,2) -- (10,2);
			\mreza{10}{4}
			\draw[fill=black] (1,1) circle(5pt);
			\draw[fill=black] (1,4) circle(5pt);
			\draw[fill=black] (10,2) circle(5pt);
			\end{tikzpicture}
			&
			\begin{tikzpicture}[thick,scale=0.5]
			\draw[line width = 2.5pt] (1,4) -- (1,3) -- (10,3) -- (10,2);
			\mreza{10}{4}
			\draw[fill=black] (1,1) circle(5pt);
			\draw[fill=black] (1,4) circle(5pt);
			\draw[fill=black] (10,2) circle(5pt);
			\end{tikzpicture}
			\\
			a & b
		\end{tabular}
		\caption{$(1,1)$,$(n,2)$-geodesic and $(1,4)$,$(n,2)$-geodesic.}
		\label{figurelemmasgePnP4easyupperbound}
	\end{figure}
\end{proof}

	By Lemma~\ref{lemmaeasyupperboundsgePnP4} and Lemma~\ref{lemmaeasylowerboundofPnPm} for $m=4$, we know that for every $n \geq 2$, the strong edge geodetic number $\sge(P_n \cp P_4)$ is either $\left\lceil 2 \sqrt{n} ~ \right\rceil$ or $\left\lceil 2 \sqrt{n} ~ \right\rceil + 1$. From Lemma~\ref{numberceil2timessqrtn} we see that Theorem~\ref{theoremsgePnP4} says that $\sge(P_n \cp P_4) = \left\lceil 2 \sqrt{n} ~ \right\rceil$ except when $n=k^2$, $n=k^2+k$, or $n=k^2+2k$ for some $k \in \N$. Hence it is enough to prove that when $n = k^2$, $n = k^2+k$, or $n=k^2+2k$ there is no strong edge geodetic set of size $\left\lceil 2 \sqrt{n} ~ \right\rceil$ and to find a strong edge geodetic set of size $\left\lceil 2 \sqrt{n} ~ \right\rceil$ in the other cases.
	
	First, let us show that there exists a strong edge geodetic set of size $\left\lceil 2 \sqrt{n} ~ \right\rceil$ for $P_n \cp P_4$ when $n = k^2 + h$ for some $k \in \N$, where $1 \leq h \leq k-1$, or $k+1 \leq h \leq 2k-1$.
	
	\begin{proposition}
		\label{propositionsgePnP4moreexamples}
		If $n=k^2+h$, $n \geq 3$, $k \geq 1$, where $1 \leq h \leq k-1$ or $k+1 \leq h \leq 2k-1$, then 
		$$\sge(P_n \cp P_4) \leq \sge(P_n \cp P_2).$$
	\end{proposition}

	\begin{proof}
	To cover all the vertical edges and the horizontal edges in the second row, we will adjust Algorithm~\ref{algorithm1verticaledgesinPnPm}. We divide this adjustment into two cases.
	
	\setcaseprenumber{1}
	\begin{case} $n = k^2 + h$, $1 \leq h \leq k-1$ (see Fig.~\ref{figurepropositionlowerboundPnP4case1example1}). \\
		First, let us use Algorithm~\ref{algorithm1verticaledgesinPnPm}.  In the next step replace the existing vertex $b_{k+1} = (k^2,1)$ with the vertex $c =(n,3)$ and for $j = h,\dots, 1$ replace the existing $a_j$,$b_{k+1}$-geodesic with two shortest paths, the $a_{j+1}$,$c$-geodesic (it exists because $h \leq k-1$) that covers the edges $(k^2+j,1)(k^2+j,2)$ and $(k^2+j,2)(k^2+j,3)$, and the $b_{j+1}$,$c$-geodesic that covers the edge $(k^2+j,3)(k^2+j,4)$. This replacement is well defined because we replaced all the shortest paths that had one endpoint at $b_{k+1}$ and added some new shortest paths that have one endpoint at $c$, so the condition that for each pair of vertices from the set $\{a_1,\dots,a_k,b_1,\dots,b_k,c\}$ we only use at most one shortest path, still holds.
		
				\begin{figure}[ht!]
			\centering
			\begin{tabular}{c c c c c}
				& & \begin{tikzpicture}[thick,scale=0.35]
				\foreach \i in {1,...,10}{
					\draw[line width = 1.5pt] (\i,1) -- (\i,4);
				};
				\draw[line width = 1.5pt] (11,1) -- (11,4);
				\draw[line width = 1.5pt] (1,1) -- (11,1);
				\draw[line width = 1.5pt] (1,4) -- (9,4);
				\draw[line width = 1.5pt] (10,4) -- (11,4);
				\mrezacrtkano{11}{4}
				\draw[fill=black] (1,1) circle(5pt);
				\draw[fill=black] (1,4) circle(5pt);
				\draw[fill=black] (4,1) circle(5pt);
				\draw[fill=black] (4,4) circle(5pt);
				\draw[fill=black] (9,1) circle(5pt);
				\draw[fill=black] (9,4) circle(5pt);
				\draw[fill=black] (11,4) circle(5pt);
				\end{tikzpicture}
				& & \\
				& & & & \\
				\adjustbox{valign=c}{
					\begin{tikzpicture}[thick,scale=0.35]
					\foreach \i in {1,...,9}{
						\draw[line width = 1.5pt] (\i,1) -- (\i,4);
					};
					\draw[line width = 1.5pt] (1,1) -- (9,1);
					\draw[line width = 1.5pt] (1,4) -- (9,4);
					\mrezacrtkano{11}{4}
					\draw[fill=black] (1,1) circle(5pt);
					\draw[fill=black] (1,4) circle(5pt);
					\draw[fill=black] (4,1) circle(5pt);
					\draw[fill=black] (4,4) circle(5pt);
					\draw[fill=black] (9,1) circle(5pt);
					\draw[fill=black] (9,4) circle(5pt);
					\draw[color=gray, fill=gray] (11,4) circle(9pt);
					\end{tikzpicture}}
				& $\to$ &
				\adjustbox{valign=c}{
					\begin{tikzpicture}[thick,scale=0.35]
					\foreach \i in {1,...,9}{
						\draw[line width = 1.5pt] (\i,1) -- (\i,4);
					};
					\draw[line width = 1.5pt] (1,1) -- (9,1);
					\draw[line width = 1.5pt] (1,4) -- (9,4);
					\mrezacrtkano{11}{4}
					\draw[fill=black] (1,1) circle(5pt);
					\draw[fill=black] (1,4) circle(5pt);
					\draw[fill=black] (4,1) circle(5pt);
					\draw[fill=black] (4,4) circle(5pt);
					\draw[fill=black] (9,1) circle(5pt);
					\draw[fill=black] (9,4) circle(5pt);
					\draw[color=gray, fill=gray] (11,3) circle(9pt);
					\end{tikzpicture}}
				& & \\
				& & & & \\
				\adjustbox{valign=c}{
					\begin{tikzpicture}[thick,scale=0.35]
					\foreach \i in {1,...,11}{
						\draw[line width = 1pt] (\i,1) -- (\i,4);
					};
					\draw[line width = 1pt] (11,1) -- (11,4);
					\draw[line width = 1pt] (1,1) -- (11,1);
					\draw[line width = 1pt] (1,4) -- (9,4);
					\draw[line width = 1pt] (10,4) -- (11,4);
					\draw[line width = 3pt] (4,1) -- (11,1) -- (11,4);
					\mrezacrtkano{11}{4}
					\draw[fill=gray, color=gray] (4,1) circle(7pt);
					\draw[fill=black] (1,4) circle(5pt);
					\draw[fill=black] (1,1) circle(5pt);
					\draw[fill=black] (4,4) circle(5pt);
					\draw[fill=black] (9,1) circle(5pt);
					\draw[fill=black] (9,4) circle(5pt);
					\draw[fill=gray, color=gray] (11,4) circle(7pt);
					\end{tikzpicture}}
				& $\to$ &
				\adjustbox{valign=c}{
					\begin{tikzpicture}[thick,scale=0.35]
					\foreach \i in {1,...,9}{
						\draw[line width = 1pt] (\i,1) -- (\i,4);
					};
					\draw[line width = 1pt] (1,1) -- (9,1);
					\draw[line width = 1pt] (1,4) -- (9,4);
					\draw[line width = 3pt] (9,1) -- (11,1) --(11,3);
					\mrezacrtkano{11}{4}
					\draw[fill=gray, color=gray] (9,1) circle(7pt);
					\draw[fill=black] (1,4) circle(5pt);
					\draw[fill=black] (4,1) circle(5pt);
					\draw[fill=black] (4,4) circle(5pt);
					\draw[fill=black] (1,1) circle(5pt);
					\draw[fill=black] (9,4) circle(5pt);
					\draw[fill=gray, color=gray] (11,3) circle(7pt);
					\end{tikzpicture}}
				& $+$ &
				\adjustbox{valign=c}{
					\begin{tikzpicture}[thick,scale=0.35]
					\foreach \i in {1,...,9}{
						\draw[line width = 1pt] (\i,1) -- (\i,4);
					};
					\draw[line width = 1pt] (11,1) -- (11,3);
					\draw[line width = 1pt] (1,1) -- (11,1);
					\draw[line width = 1pt] (1,4) -- (9,4);
					\draw[line width = 3pt] (9,4) -- (11,4) --(11,3);
					\mrezacrtkano{11}{4}
					\draw[fill=gray, color=gray] (9,4) circle(7pt);
					\draw[fill=black] (1,4) circle(5pt);
					\draw[fill=black] (4,1) circle(5pt);
					\draw[fill=black] (4,4) circle(5pt);
					\draw[fill=black] (1,1) circle(5pt);
					\draw[fill=black] (9,1) circle(5pt);
					\draw[fill=gray, color=gray] (11,3) circle(7pt);
					\end{tikzpicture}}
				\\
				& & & & \\
				\adjustbox{valign=c}{
					\begin{tikzpicture}[thick,scale=0.35]
					\foreach \i in {1,...,11}{
						\draw[line width = 1pt] (\i,1) -- (\i,4);
					};
					\draw[line width = 1pt] (11,1) -- (11,4);
					\draw[line width = 1pt] (1,1) -- (11,1);
					\draw[line width = 1pt] (1,4) -- (9,4);
					\draw[line width = 1pt] (10,4) -- (11,4);
					\draw[line width = 3pt] (1,1) -- (10,1) --(10,4) -- (11,4);
					\mrezacrtkano{11}{4}
					\draw[fill=gray, color=gray] (1,1) circle(7pt);
					\draw[fill=black] (1,4) circle(5pt);
					\draw[fill=black] (4,1) circle(5pt);
					\draw[fill=black] (4,4) circle(5pt);
					\draw[fill=black] (9,1) circle(5pt);
					\draw[fill=black] (9,4) circle(5pt);
					\draw[fill=gray, color=gray] (11,4) circle(7pt);
					\end{tikzpicture}}
				& $\to$ &
				\adjustbox{valign=c}{
					\begin{tikzpicture}[thick,scale=0.35]
					\foreach \i in {1,...,9}{
						\draw[line width = 1pt] (\i,1) -- (\i,4);
					};
					\draw[line width = 1pt] (11,1) -- (11,4);
					\draw[line width = 1pt] (1,1) -- (11,1);
					\draw[line width = 1pt] (1,4) -- (11,4);
					\draw[line width = 1pt] (10,4) -- (11,4);
					\draw[line width = 3pt] (4,1) -- (10,1) --(10,3) -- (11,3);
					\mrezacrtkano{11}{4}
					\draw[fill=gray, color=gray] (4,1) circle(7pt);
					\draw[fill=black] (1,4) circle(5pt);
					\draw[fill=black] (1,1) circle(5pt);
					\draw[fill=black] (4,4) circle(5pt);
					\draw[fill=black] (9,1) circle(5pt);
					\draw[fill=black] (9,4) circle(5pt);
					\draw[fill=gray, color=gray] (11,3) circle(7pt);
					\end{tikzpicture}}
				& $+$ &
				\adjustbox{valign=c}{
					\begin{tikzpicture}[thick,scale=0.35]
					\foreach \i in {1,...,11}{
						\draw[line width = 1pt] (\i,1) -- (\i,4);
					};
					\draw[line width = 1pt] (11,1) -- (11,4);
					\draw[line width = 1pt] (1,1) -- (11,1);
					\draw[line width = 1pt] (1,4) -- (9,4);
					\draw[line width = 1pt] (10,4) -- (11,4);
					\draw[line width = 3pt] (4,4) -- (10,4) --(10,3) -- (11,3);
					\mrezacrtkano{11}{4}
					\draw[fill=gray, color=gray] (4,4) circle(7pt);
					\draw[fill=black] (1,1) circle(5pt);
					\draw[fill=black] (4,1) circle(5pt);
					\draw[fill=black] (1,4) circle(5pt);
					\draw[fill=black] (9,1) circle(5pt);
					\draw[fill=black] (9,4) circle(5pt);
					\draw[fill=gray, color=gray] (11,3) circle(7pt);
					\end{tikzpicture}}
				\\
				& & & & \\
				& & \begin{tikzpicture}[thick,scale=0.35]
				\foreach \i in {1,...,11}{
					\draw[line width = 1pt] (\i,1) -- (\i,4);
				};
				\draw[line width = 1pt] (11,1) -- (11,4);
				\draw[line width = 1pt] (1,1) -- (11,1);
				\draw[line width = 1pt] (1,4) -- (10,4);
				\draw[line width = 1pt] (10,4) -- (11,4);
				\draw[line width = 3pt] (1,4) -- (1,3) --(11,3);
				\mrezacrtkano{11}{4}
				\draw[fill=gray, color=gray] (1,4) circle(7pt);
				\draw[fill=black] (1,1) circle(5pt);
				\draw[fill=black] (4,1) circle(5pt);
				\draw[fill=black] (4,4) circle(5pt);
				\draw[fill=black] (9,1) circle(5pt);
				\draw[fill=black] (9,4) circle(5pt);
				\draw[fill=gray, color=gray] (11,3) circle(7pt);
				\end{tikzpicture}
				& & \\
				& & & & \\
				& & \begin{tikzpicture}[thick,scale=0.35]
				\foreach \i in {1,...,11}{
					\draw[line width = 1pt] (\i,1) -- (\i,4);
				};
				\draw[line width = 1pt] (11,1) -- (11,4);
				\draw[line width = 1pt] (1,1) -- (11,1);
				\draw[line width = 1pt] (1,4) -- (11,4);
				\draw[line width = 1pt] (1,3) -- (11,3);
				\draw[line width = 3pt] (1,1) -- (1,2) -- (11,2) -- (11,3);
				\mrezacrtkano{11}{4}
				\draw[fill=gray, color=gray] (1,1) circle(7pt);
				\draw[fill=black] (1,4) circle(5pt);
				\draw[fill=black] (4,1) circle(5pt);
				\draw[fill=black] (4,4) circle(5pt);
				\draw[fill=black] (9,1) circle(5pt);
				\draw[fill=black] (9,4) circle(5pt);
				\draw[fill=gray, color=gray] (11,3) circle(7pt);
				\end{tikzpicture}
				& & \\
			\end{tabular}
			\caption{Strong edge geodetic set in $P_{3^2+2} \cp P_4$.}
			\label{figurepropositionlowerboundPnP4case1example1}
		\end{figure}
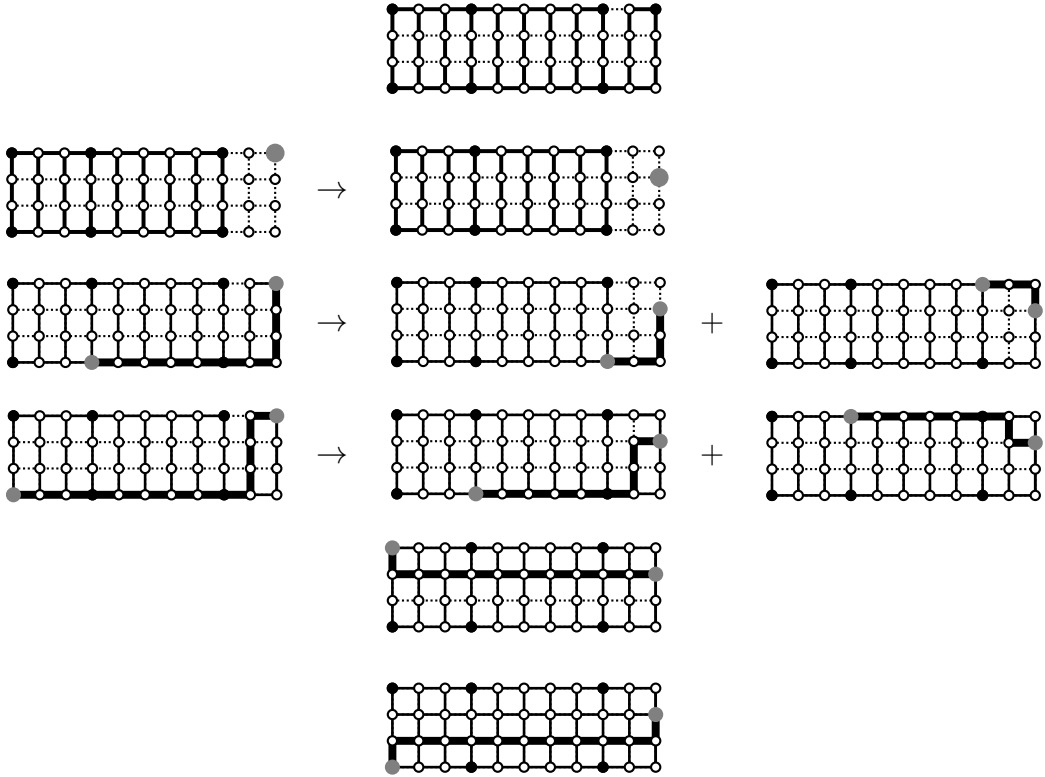
		In this way we have adjusted Algorithm~\ref{algorithm1verticaledgesinPnPm} such that it covers all vertical edges in $P_n \cp P_4$, while it uses neither the $a_1$,$c$-geodesics nor the $b_1$,$c$-geodesics. This means that we can add the $a_1$,$c$-geodesic that covers all the horizontal edges in the second row and the $b_1$,$c$-geodesic that covers all the horizontal edges in the third row. Some edges from the first row are already covered with the $a_{h+1}$,$c$-geodesic, but the rest of them can be covered with the unique $a_1$,$a_k$-geodesic. By symmetry, some of the horizontal edges from the fourth row are covered with the $b_{h+1}$,$c$-geodesic, and the other ones are covered by the $b_1$,$b_k$-geodesic.
	\end{case}
	
	\begin{case} $n = k^2 + h$, $k+1 \leq h \leq 2k-1$ (see Fig.~\ref{figurepropositionlowerboundPnP4case2example1}). \\
		In this case, because $h \leq 2k-1$, Algorithm~\ref{algorithm1verticaledgesinPnPm} never uses a $a_{k+1}$,$b_k$-geodesic. It also never uses the unique $a_{k+1}$,$b_{k+1}$-geodesic. If we add this shortest path, we can replace the existing $b_{h-k}$,$a_{k+1}$-geodesic (this shortest path covers the vertical edges in the $n$-th column) with the one that covers the vertical edges in the $(k(k+1)+1)$-th column (in Algorithm~\ref{algorithm1verticaledgesinPnPm} covered by the $a_{k+1}$,$b_1$-geodesic). Observe that this step does nothing if $h=1$. This way we can replace the existing $a_{k+1}$,$b_1$-geodesic with the one that covers all the horizontal edges in the third row.
		
		Also, if we add the $a_{k+1}$,$b_k$-geodesic that covers all the vertical edges in the $(k^2+1)$-th column, we can replace the existing $a_1$,$b_{k+1}$-geodesic with the one that covers all the horizontal edges in the second row.
		
		\begin{figure}[ht!]
			\centering
			\begin{tikzpicture}[thick,scale=0.35]
			\foreach \i in {1,...,14}{
				\draw[line width = 1pt] (\i,1) -- (\i,4);
			};
			\draw[line width = 1pt] (1,1) -- (12,1);
			\draw[line width = 1pt] (13,1) -- (14,1);
			\draw[line width = 1pt] (1,4) -- (14,4);
			\mrezacrtkano{14}{4}
			\draw[fill=black] (1,1) circle(5pt);
			\draw[fill=black] (1,4) circle(5pt);
			\draw[fill=black] (4,1) circle(5pt);
			\draw[fill=black] (4,4) circle(5pt);
			\draw[fill=black] (9,1) circle(5pt);
			\draw[fill=black] (9,4) circle(5pt);
			\draw[fill=black] (14,1) circle(5pt);
			\draw[fill=black] (14,4) circle(5pt);
			\end{tikzpicture}
			\\
			\vspace{0.5cm}
			\begin{tikzpicture}[thick,scale=0.35]
			\foreach \i in {1,...,14}{
				\draw[line width = 1pt] (\i,1) -- (\i,4);
			};
			\draw[line width = 1pt] (1,1) -- (12,1);
			\draw[line width = 1pt] (13,1) -- (14,1);
			\draw[line width = 1pt] (1,4) -- (14,4);
			\draw[line width = 3pt] (14,1) -- (14,4);
			\mrezacrtkano{14}{4}
			\draw[fill=black] (1,1) circle(5pt);
			\draw[fill=black] (1,4) circle(5pt);
			\draw[fill=black] (4,1) circle(5pt);
			\draw[fill=black] (4,4) circle(5pt);
			\draw[fill=black] (9,1) circle(5pt);
			\draw[fill=black] (9,4) circle(5pt);
			\draw[fill=gray, color=gray] (14,1) circle(7pt);
			\draw[fill=gray, color=gray] (14,4) circle(7pt);
			\end{tikzpicture}
			\\
			\vspace{0.5cm}
			\begin{tabular}{c c c}
				\adjustbox{valign=c}{
					\begin{tikzpicture}[thick,scale=0.35]
					\foreach \i in {1,...,14}{
						\draw[line width = 1pt] (\i,1) -- (\i,4);
					};
					\draw[line width = 1pt] (1,1) -- (12,1);
					\draw[line width = 1pt] (13,1) -- (14,1);
					\draw[line width = 1pt] (1,4) -- (14,4);
					\draw[line width = 3pt] (4,4) -- (14,4) -- (14,1);
					\mrezacrtkano{14}{4}
					\draw[fill=black] (1,1) circle(5pt);
					\draw[fill=black] (1,4) circle(5pt);
					\draw[fill=black] (4,1) circle(5pt);
					\draw[fill=black] (9,4) circle(5pt);
					\draw[fill=black] (9,1) circle(5pt);
					\draw[fill=gray, color=gray] (4,4) circle(7pt);
					\draw[fill=gray, color=gray] (14,1) circle(7pt);
					\draw[fill=black] (14,4) circle(5pt);
					\end{tikzpicture}}
				& $\to$ &
				\adjustbox{valign=c}{
					\begin{tikzpicture}[thick,scale=0.35]
					\foreach \i in {1,...,14}{
						\draw[line width = 1pt] (\i,1) -- (\i,4);
					};
					\draw[line width = 1pt] (1,1) -- (12,1);
					\draw[line width = 1pt] (13,1) -- (14,1);
					\draw[line width = 1pt] (1,4) -- (14,4);
					\draw[line width = 3pt] (4,4) -- (13,4) -- (13,1) -- (14,1);
					\mrezacrtkano{14}{4}
					\draw[fill=black] (1,1) circle(5pt);
					\draw[fill=black] (1,4) circle(5pt);
					\draw[fill=black] (4,1) circle(5pt);
					\draw[fill=black] (9,4) circle(5pt);
					\draw[fill=black] (9,1) circle(5pt);
					\draw[fill=gray, color=gray] (4,4) circle(7pt);
					\draw[fill=gray, color=gray] (14,1) circle(7pt);
					\draw[fill=black] (14,4) circle(5pt);
					\end{tikzpicture}}
				\\
			\end{tabular}
			\\
			\vspace{0.5cm}
			\begin{tabular}{c c c}
				\adjustbox{valign=c}{
					\begin{tikzpicture}[thick,scale=0.35]
					\foreach \i in {1,...,14}{
						\draw[line width = 1pt] (\i,1) -- (\i,4);
					};
					\draw[line width = 1pt] (1,1) -- (12,1);
					\draw[line width = 1pt] (13,1) -- (14,1);
					\draw[line width = 1pt] (1,4) -- (14,4);
					\draw[line width = 3pt] (1,4) -- (13,4) -- (13,1) -- (14,1);
					\mrezacrtkano{14}{4}
					\draw[fill=black] (1,1) circle(5pt);
					\draw[fill=black] (1,4) circle(5pt);
					\draw[fill=black] (4,1) circle(5pt);
					\draw[fill=black] (9,4) circle(5pt);
					\draw[fill=black] (9,1) circle(5pt);
					\draw[fill=gray, color=gray] (1,4) circle(7pt);
					\draw[fill=gray, color=gray] (14,1) circle(7pt);
					\draw[fill=black] (14,4) circle(5pt);
					\end{tikzpicture}}
				& $\to$ &
				\adjustbox{valign=c}{
					\begin{tikzpicture}[thick,scale=0.35]
					\foreach \i in {1,...,14}{
						\draw[line width = 1pt] (\i,1) -- (\i,4);
					};
					\draw[line width = 1pt] (1,1) -- (12,1);
					\draw[line width = 1pt] (13,1) -- (14,1);
					\draw[line width = 1pt] (1,4) -- (14,4);
					\draw[line width = 3pt] (1,4) -- (1,3) -- (14,3) -- (14,1);
					\mrezacrtkano{14}{4}
					\draw[fill=black] (1,1) circle(5pt);
					\draw[fill=black] (1,4) circle(5pt);
					\draw[fill=black] (4,1) circle(5pt);
					\draw[fill=black] (9,4) circle(5pt);
					\draw[fill=black] (9,1) circle(5pt);
					\draw[fill=gray, color=gray] (1,4) circle(7pt);
					\draw[fill=gray, color=gray] (14,1) circle(7pt);
					\draw[fill=black] (14,4) circle(5pt);
					\end{tikzpicture}}
				\\
			\end{tabular}
			\\
			\vspace{0.5cm}
			\begin{tikzpicture}[thick,scale=0.35]
			\foreach \i in {1,...,14}{
				\draw[line width = 1pt] (\i,1) -- (\i,4);
			};
			\draw[line width = 1pt] (1,1) -- (12,1);
			\draw[line width = 1pt] (13,1) -- (14,1);
			\draw[line width = 1pt] (1,3) -- (14,3);
			\draw[line width = 1pt] (1,4) -- (14,4);
			\draw[line width = 3pt] (9,4) -- (10,4) -- (10,1) -- (14,1);
			\mrezacrtkano{14}{4}
			\draw[fill=black] (1,1) circle(5pt);
			\draw[fill=black] (1,4) circle(5pt);
			\draw[fill=black] (4,1) circle(5pt);
			\draw[fill=black] (4,4) circle(5pt);
			\draw[fill=black] (9,1) circle(5pt);
			\draw[fill=black] (14,4) circle(5pt);
			\draw[fill=gray, color=gray] (14,1) circle(7pt);
			\draw[fill=gray, color=gray] (9,4) circle(7pt);
			\end{tikzpicture}
			\\
			\vspace{0.5cm}
			\begin{tabular}{c c c}
				\adjustbox{valign=c}{
					\begin{tikzpicture}[thick,scale=0.35]
					\foreach \i in {1,...,14}{
						\draw[line width = 1pt] (\i,1) -- (\i,4);
					};
					\draw[line width = 1pt] (1,1) -- (14,1);
					\draw[line width = 1pt] (1,3) -- (14,3);
					\draw[line width = 1pt] (1,4) -- (14,4);
					\draw[line width = 3pt] (1,1) -- (10,1) -- (10,4) -- (14,4);
					\mrezacrtkano{14}{4}
					\draw[fill=gray, color=gray] (1,1) circle(7pt);
					\draw[fill=black] (1,4) circle(5pt);
					\draw[fill=black] (4,1) circle(5pt);
					\draw[fill=black] (4,4) circle(5pt);
					\draw[fill=black] (9,1) circle(5pt);
					\draw[fill=black] (9,4) circle(5pt);
					\draw[fill=gray, color=gray] (14,4) circle(7pt);
					\draw[fill=black] (14,1) circle(5pt);
					\end{tikzpicture}}
				& $\to$ &
				\adjustbox{valign=c}{
					\begin{tikzpicture}[thick,scale=0.35]
					\foreach \i in {1,...,14}{
						\draw[line width = 1pt] (\i,1) -- (\i,4);
					};
					\draw[line width = 1pt] (1,1) -- (14,1);
					\draw[line width = 1pt] (1,3) -- (14,3);
					\draw[line width = 1pt] (1,4) -- (14,4);
					\draw[line width = 3pt] (1,1) -- (1,2) -- (14,2) -- (14,4);
					\mrezacrtkano{14}{4}
					\draw[fill=gray, color=gray] (1,1) circle(7pt);
					\draw[fill=black] (1,4) circle(5pt);
					\draw[fill=black] (4,1) circle(5pt);
					\draw[fill=black] (4,4) circle(5pt);
					\draw[fill=black] (9,1) circle(5pt);
					\draw[fill=black] (9,4) circle(5pt);
					\draw[fill=gray, color=gray] (14,4) circle(7pt);
					\draw[fill=black] (14,1) circle(5pt);
					\end{tikzpicture}}
			\end{tabular}
			\caption{Strong edge geodetic set in $P_{3^2+2 \cdot 3 - 1} \cp P_4$.}
			\label{figurepropositionlowerboundPnP4case2example1}
		\end{figure}
		In the same way as in $P_n \cp P_2$ we now cover the horizontal edges in the first row (using the $a_1$,$a_{k+1}$-geodesic) and in the fourth row (using the $b_1$,$b_{k+1}$-geodesic) of $P_n \cp P_4$.
	\end{case}
	Since in all three cases we found a strong edge geodetic set of size $\sge(P_n \cp P_2)$, we can conclude that $\sge(P_n \cp P_4) \leq \sge(P_n \cp P_2)$ for all $n \in \N$ when $n \not = k^2$, $n \not= k^2-1$, and $n \not= k^2+k$ for some $k \in \N$.
	\end{proof}

	We will now show that there is no strong edge geodetic set of size $\left\lceil 2 \sqrt{n} \right\rceil$ for $P_n \cp P_4$ when $n = k^2 -1$, $k^2$, or $k^2+k$ for some $k \in \N$.
		
	\begin{proposition}
		\label{propositionsgePnP4specialexamples}
		If $n = k^2$, $n = k^2 -1$, or $n = k^2 + k$ for some $k \geq 2$, $k \in \N$, then
		$$\sge(P_n \cp P_4) > \sge(P_n \cp P_2).$$
	\end{proposition}
	
	\begin{proof}
	Let us define the function
	$$f_s (a,b,c,d) = a b + b c +c d + 2ac + 2bd + 3ad$$
	which gives an upper bound on how many different vertical edges can be covered with shortest paths between $s$ vertices, where $a$ vertices are in the first row, $b$ vertices in the second row, $c$ vertices in the third row, and $d$ vertices in the fourth row of $P_n \cp P_4$. The maximum values of function $f_s$ under some bounds are again computed by computer and are gathered in Table~\ref{tablemaxfsPnP4}.
	
\begin{table}[ht!]
	\centering
	\begin{tabular}{c c | c c}
		conditions & & & upper bound \\
		\hline
		& & & \\
		$a,b,c,d \geq 0$ & & & $f_s(a,b,c,d) \leq {1 \over 4} (3s^2)$ \\
		& & & \\
		$a,c,d \geq 0$, & $b \geq 1$ & & $f_s(a,b,c,d) \leq {1 \over 12} (9s^2-8)$ \\
		& & & \\
		$a,b,c,d \geq 0$, & $b+c \geq 2$ & & $f_s(a,b,c,d) \leq {1 \over 4} (3s^2-8)$ \\
		& & & \\
		$a,b,c,d \geq 0$, & $b+c \geq 3$ & & $f_s(a,b,c,d) \leq {1 \over 4} (3s^2-18)$ \\
	\end{tabular}
	\caption{Maximal values of function $f_s(a,b,c,d)$.}
	\label{tablemaxfsPnP4}
\end{table}
	
	Every strong edge geodetic set $S$ has to include at least one vertex from the first column, otherwise it would not be possible to cover the edge $(1,1)(1,2)$. Similarly, every strong edge geodetic set includes at least one vertex from the last column. Depending on a strong edge geodetic set $S$ in $P_n \cp P_4$ we will define a type of first or last column of vertices. All the types for a strong edge geodetic set with at least one vertex in the first and at least one vertex in the last column are gathered in Table~\ref{tableoftypes}. For example, if $(1,1), (1,2) \in S$ and $(1,3),(1,4) \not\in S$, we say that the first column of vertices in $P_n \cp P_4$ is of type $D$. Symmetrically, if $(1,1), (1,2) \not\in S$ and $(1,3),(1,4) \in S$, we also say that the first column of vertices in $P_n \cp P_4$ is of type $D$. We symmetrically define the type of the last column.
	
	\begin{table}[ht!]
		\centering
		\begin{tabular}{ c | c c | c }
			Type $T$ & type in first column & & $r(T)$ \\
			\hline
			& & & \\
			$A$
			&
			\adjustbox{valign=c}{
				\begin{tikzpicture}[thick,scale=0.45]
				\mreza{10}{4}
				\draw[fill=black] (1,1) circle(5pt);
				\draw[fill=black] (1,2) circle(5pt);
				\draw[fill=black] (1,3) circle(5pt);
				\draw[fill=black] (1,4) circle(5pt);
				\end{tikzpicture}}
			& &
			$7$ \\
			& & &\\
			$B$
			&
			\adjustbox{valign=c}{
				\begin{tikzpicture}[thick,scale=0.45]
				\mreza{10}{4}
				\draw[fill=black] (1,1) circle(5pt);
				\draw[fill=black] (1,2) circle(5pt);
				\draw[fill=black] (1,3) circle(5pt);
				\end{tikzpicture}}
			&
			\adjustbox{valign=c}{
				\begin{tikzpicture}[thick,scale=0.45]
				\mreza{10}{4}
				\draw[fill=black] (1,2) circle(5pt);
				\draw[fill=black] (1,3) circle(5pt);
				\draw[fill=black] (1,4) circle(5pt);
				\end{tikzpicture}}
			&
			$2$ \\
			& & & \\
			$C$
			&
			\adjustbox{valign=c}{
				\begin{tikzpicture}[thick,scale=0.45]
				\mreza{10}{4}
				\draw[fill=black] (1,1) circle(5pt);
				\draw[fill=black] (1,2) circle(5pt);
				\draw[fill=black] (1,4) circle(5pt);
				\end{tikzpicture}}
			&
			\adjustbox{valign=c}{
				\begin{tikzpicture}[thick,scale=0.45]
				\mreza{10}{4}
				\draw[fill=black] (1,1) circle(5pt);
				\draw[fill=black] (1,3) circle(5pt);
				\draw[fill=black] (1,4) circle(5pt);
				\end{tikzpicture}}
			&
			$4$ \\
			& & & \\
			$D$
			&
			\adjustbox{valign=c}{
				\begin{tikzpicture}[thick,scale=0.45]
				\mreza{10}{4}
				\draw[fill=black] (1,1) circle(5pt);
				\draw[fill=black] (1,2) circle(5pt);
				\end{tikzpicture}}
			&
			\adjustbox{valign=c}{
				\begin{tikzpicture}[thick,scale=0.45]
				\mreza{10}{4}
				\draw[fill=black] (1,3) circle(5pt);
				\draw[fill=black] (1,4) circle(5pt);
				\end{tikzpicture}}
			&
			$1$ \\
			& & & \\
			$E$
			&
			\adjustbox{valign=c}{
				\begin{tikzpicture}[thick,scale=0.45]
				\mreza{10}{4}
				\draw[fill=black] (1,1) circle(5pt);
				\draw[fill=black] (1,3) circle(5pt);
				\end{tikzpicture}}
			&
			\adjustbox{valign=c}{
				\begin{tikzpicture}[thick,scale=0.45]
				\mreza{10}{4}
				\draw[fill=black] (1,2) circle(5pt);
				\draw[fill=black] (1,4) circle(5pt);
				\end{tikzpicture}}
			&
			$1$ \\
			& & & \\
			$F$
			&
			\adjustbox{valign=c}{
				\begin{tikzpicture}[thick,scale=0.45]
				\mreza{10}{4}
				\draw[fill=black] (1,1) circle(5pt);
				\draw[fill=black] (1,4) circle(5pt);
				\end{tikzpicture}}
			& &
			$2$ \\
			& & & \\
			$G$
			&
			\adjustbox{valign=c}{
				\begin{tikzpicture}[thick,scale=0.45]
				\mreza{10}{4}
				\draw[fill=black] (1,2) circle(5pt);
				\draw[fill=black] (1,3) circle(5pt);
				\end{tikzpicture}}
			& &
			$0$
			\\
			& & & \\
			$H$
			&
			\adjustbox{valign=c}{
				\begin{tikzpicture}[thick,scale=0.45]
				\mreza{10}{4}
				\draw[fill=black] (1,1) circle(5pt);
				\end{tikzpicture}}
			&
			\adjustbox{valign=c}{
				\begin{tikzpicture}[thick,scale=0.45]
				\mreza{10}{4}
				\draw[fill=black] (1,4) circle(5pt);
				\end{tikzpicture}}
			&
			$3$ \\
			& & & \\
			$I$
			&
			\adjustbox{valign=c}{
				\begin{tikzpicture}[thick,scale=0.45]
				\mreza{10}{4}
				\draw[fill=black] (1,2) circle(5pt);
				\end{tikzpicture}}
			&
			\adjustbox{valign=c}{
				\begin{tikzpicture}[thick,scale=0.45]
				\mreza{10}{4}
				\draw[fill=black] (1,3) circle(5pt);
				\end{tikzpicture}}
			&
			$1$ \\
			& & & \\
		\end{tabular}
		\caption{Different types of a strong edge geodetic set in the first or the last column.}
		\label{tableoftypes}
	\end{table}
	
	For a shortest path $P$ in $P_n \cp P_4$, let $E_i(P)$ denote the set of edges in $i$-th column, that is,
	$$E_i(P) =  E(P) \cap \{(i,1)(i,2), (i,2)(i,3), (i,3)(i,4)\}.$$
	For a strong edge geodetic covering $C = \{ P_{x,y}\}$ set
	$$r_C^1 = \sum_{\{x,y\} \in \binom{S}{2}} |E_1(P_{x,y})|-3.$$
	Roughly speaking, $r_C^1$ measures redundancy of the strong edge geodetic covering in the first column of $P_n \cp P_4$. Analogously, the redundancy $r_C^n$ with respect to the last column is introduced.
	
	For each type $T$ of the first column determined by $S$, we can now compute the minimum number of redundant coverings in the first column of edges as
	$r_1(T) = \min_{C}\{r_C^1\}$,
	where $C$ is a strong edge geodetic covering for a strong edge geodetic set $S$, where the first column in $P_n \cp P_4$ is of type $T$. Similarly we can define the minimum number of redundant coverings in the last column of edges, $r_n(T)$. By symmetry, these two numbers are the same, so we can denote $r(T) = r_1(T) = r_n(T)$. For each type $T$, the number $r(T)$ is listed in the last column of Table~\ref{tableoftypes}.
	
	We will show how to compute $r(C)$ for type $C$, for other types it is similar. We can without loss of generality assume $\{(1,1), (1,2), (1,4)\} \subset S$. Between the pairs of these three vertices, there are three unique shortest paths $P_1$, $P_2$, and $P_3$. To cover the horizontal edge $(1,3)(2,3)$ we need another shortest path $P_4$ that has one endvertex in the vertex set $\{(1,1), (1,2), (1,4)\}$. Because $(1,3) \not\in S$, this path includes at least one vertical edge from the first column, which implies
	$$\sum_{\{x,y\} \in \binom{S}{2}} |E_1(P_{x,y})|-3 \geq |E_1(P_1)| + |E_1(P_2)| + |E_1(P_3)| + |E_1(P_4)| - 3 \geq 3.$$
	
	Suppose now that $S$ is a strong edge geodetic set for $P_n \cp P_4$ of cardinality $|S| = \sge(P_n \cp P_2)$. We distinguish four different cases.
	
	\setcaseprenumber{1}
	\begin{case} $S$ does not contain any vertex from the set $\{(1,2),(1,3),(n,2),(n,3)\}$. \\
		In this case the first and the last column are either of type $F$ or type $G$, which implies that $r_1+r_n \geq 4$. Because $f_s(a,b,c,d) \leq {3s^2 \over 4}$ when $a,b,c,d \geq 0$, it holds
		\begin{eqnarray*}
		f_{s}(a,b,c,d)-4 &\leq& {3 \cdot (2k)^2 \over 4}-4 = 3k^2-4 < 3n \text{ for } n=k^2-1; \\
		f_{s}(a,b,c,d)-4 &\leq& {3 \cdot (2k)^2 \over 4}-4 = 3k^2-4 < 3n \text{ for } n=k^2; \\
		f_{s}(a,b,c,d)-4 &\leq& {3 \cdot (2k+1)^2 \over 4}-4 = 3k^2+3k - {13 \over 4} < 3n \text{ for } n=k^2+k, \\
		\end{eqnarray*}
		
		\vspace{-0.4cm}
		\noindent for $s=\sge(P_n \cp P_2)$.
	\end{case}
	
	\begin{case} $S$ contains at least three vertices from the second and the third row. \\
	Because $f_s(a,b,c,d) \leq {1 \over 4} (3s^2-18)$ when $a,b,c,d \geq 0$, $b+c \geq 3$, it holds
		\begin{eqnarray*}
		f_{s}(a,b,c,d) &\leq& {1 \over 4} (3(2k)^2-18) = 3k^2-{18 \over 4} < 3n \text{ for } n=k^2-1; \\
		f_{s}(a,b,c,d) &\leq& {1 \over 4} (3(2k)^2-18) = 3k^2- {18 \over 4} < 3n \text{ for } n=k^2; \\
		f_{s}(a,b,c,d) &\leq& {1 \over 4} (3(2k+1)^2-18) = 3k^2+3k-{15 \over 4} < 3n \text{ for } n=k^2+k, \\
		\end{eqnarray*}
		
		\vspace{-0.4cm}
		\noindent for $s=\sge(P_n \cp P_2)$.
	\end{case}
	
	\begin{case} $S$ contains exactly two vertices from the set $\{(1,2),(1,3),(n,2),(n,3)\}$ and no other vertex from the second and the third row. \\
		If we look at all types of the first and the last column such that the condition holds, we see that $r_1 + r_n$ is always at least $2$ (observe that type $G$ can only be combined with types $F$ and $H$). Because $f_s(a,b,c,d) \leq {1 \over 4} (3s^2-8)$ when $b+c \geq 2$, and $a,b,c,d \geq 0$, it holds
		\begin{eqnarray*}
		f_{s}(a,b,c,d) - 2 &\leq& {1 \over 4} (3(2k)^2-8) - 2 = 3k^2-4 < 3n \text{ for } n=k^2-1; \\
		f_{s}(a,b,c,d) - 2 &\leq& {1 \over 4} (3(2k)^2-8) - 2 = 3k^2-4 < 3n \text{ for } n=k^2; \\
		f_{s}(a,b,c,d) - 2 &\leq& {1 \over 4} (3(2k+1)^2-8) - 2 = 3k^2+3k-{13 \over 4} < 3n \text{ for } n=k^2+k, \\
		\end{eqnarray*}
		
		\vspace{-0.4cm}
		\noindent for $s=\sge(P_n \cp P_2)$.  
	\end{case}
	
	\begin{case} $S$ contains exactly one vertex from the set $\{(1,2),(1,3),(n,2),(n,3)\}$ and no other vertex from the second and the third row. \\
		By symmetry, we can without loss of generality assume that $(1,3),(n,2),(n,3) \not\in S$ and $(1,2) \in S$. This implies that the first column is of type $C$, $D$, $E$, or $H$ and the last column is of type $F$ or $G$. The sum of the number of redundant coverings for the first and the last column is than at least $3$. Because $f_s(a,b,c,d) \leq {1 \over 12} (9s^2-8)$ when $b+c \geq 2$ and $a,b,c,d \geq 0$, it holds
		\begin{eqnarray*}
		f_{s}(a,b,c,d) - 2 &\leq& {1 \over 12} (9(2k)^2-8) - 3 = 3k^2-{11 \over 3} < 3n \text{ for } n=k^2-1; \\
		f_{s}(a,b,c,d) - 2 &\leq& {1 \over 12} (9(2k)^2-8) - 3 = (3k^2-3)-{2 \over 3} < 3n \text{ for } n=k^2; \\
		f_{s}(a,b,c,d) - 2 &\leq& {1 \over 12} (9(2k+1)^2-8) - 3 = (3k^2+3k)-{37 \over 12} < 3n \text{ for } n=k^2+k, \\
		\end{eqnarray*}
		
		\vspace{-0.4cm}
		\noindent for $s=\sge(P_n \cp P_2)$.
	\end{case}
	
	Since $3n$ is the number of all vertical edges in $P_n \cp P_4$, it holds $|S| > \sge(P_n \cp P_2)$ in every case above. This is a contradiction with the assumption that $S$ is a strong edge geodetic set of $P_n \cp P_4$ of cardinality $\sge(P_n \cp P_2)$.
	
	\newpage
	
	Because all four cases led us to a contradiction, we can conclude that for $P_n \cp P_4$, where $n=k^2-1, k^2$, or $k^2+k$ for some $k \in \N$, there is no strong edge geodetic set of size $\sge(P_n \cp P_2)$. By Lemma~\ref{lemmaeasyupperboundsgePnP4} this means that $\sge(P_n \cp P_4) = \sge(P_n \cp P_2) + 1$ when $n=k^2-1, k^2$, or $k^2+k$ for some $k \in \N$. In other words, $\sge(P_{k^2-1} \cp P_4) = \sge(P_{k^2} \cp P_4) = 2k+1$ and $\sge(P_{k^2+k} \cp P_4) = 2k+2$. Also, because $k^2-1 = (k-1)^2 + 2(k-1)$, we have $\sge(P_{k^2+2k} \cp P_4) = 2k+3$.
	\end{proof}

\section{Upper bounds }
\label{sec:upperbounds}

In this concluding section we give two upper bounds on $\sge(P_n \cp P_m)$.

\begin{proposition}
\label{propositiongeneralboundsgePnP2+sgePm-2P2}
	If $n \geq 2$ and $m \geq 2$, then
	$$\sge(P_n \cp P_m) \leq \left\lceil 2 \sqrt{n} ~ \right\rceil + \left\lceil 2 \sqrt{m-2} ~ \right\rceil.$$
\end{proposition}

\begin{proof}
	First we cover all the vertical edges in $P_n \cp P_m$ by using Algorithm~\ref{algorithm1verticaledgesinPnPm}. This algorithm uses $\left\lceil 2 \sqrt{n} ~ \right\rceil$ vertices from the first and the last row. We also, similarly as in $P_n \cp P_2$, use these vertices to cover the horizontal edges from the first and the last row (Fig.~\ref{figurepropositiongeneralboundsgePnP2+sgePm-2P2}a).
	
	\begin{figure}[ht!]
	\centering
	\begin{tabular}{c c c}
	\begin{tikzpicture}[thick,scale=0.36]
		\foreach \i in {1,...,14}{
			\draw[line width = 1.5pt] (\i,1) -- (\i,8);
		};
		\draw[line width = 1.5pt] (1,1) -- (14,1);
		\draw[line width = 1.5pt] (1,8) -- (14,8);
		\mrezacrtkano{14}{8}
		\draw[fill=black] (1,1) circle(5pt);
		\draw[fill=black] (1,8) circle(5pt);
		\draw[fill=black] (4,1) circle(5pt);
		\draw[fill=black] (4,8) circle(5pt);
		\draw[fill=black] (9,1) circle(5pt);
		\draw[fill=black] (9,8) circle(5pt);
		\draw[fill=black] (14,1) circle(5pt);
		\draw[fill=black] (14,8) circle(5pt);
	\end{tikzpicture}
	& &
	\begin{tikzpicture}[thick,scale=0.36]
		\draw[line width = 1pt, color=gray] (1,0) -- (14,0);
		\draw[line width = 1pt, color=gray] (1,7) -- (14,7);
		\foreach \i in {1,...,14}{
			\draw[line width = 2pt] (\i,1) -- (\i,6);
			\draw[line width = 1pt, color=gray] (\i,0) --(\i,1);
			\draw[line width = 1pt, color=gray] (\i,6) --(\i,7);
			\draw[color=gray, fill=white] (\i,0) circle(5pt);
			\draw[color=gray, fill=white] (\i,7) circle(5pt);
		};
		\mrezacrtkano{14}{6}
		\draw[color=gray, fill=gray] (1,0) circle(5pt);
		\draw[color=gray, fill=gray] (1,7) circle(5pt);
		\draw[color=gray, fill=gray] (4,0) circle(5pt);
		\draw[color=gray, fill=gray] (4,7) circle(5pt);
		\draw[color=gray, fill=gray] (9,0) circle(5pt);
		\draw[color=gray, fill=gray] (9,7) circle(5pt);
		\draw[color=gray, fill=gray] (14,0) circle(5pt);
		\draw[color=gray, fill=gray] (14,7) circle(5pt);
	\end{tikzpicture}
	\\
	a && b \\
	& & \\
	\begin{tikzpicture}[thick,scale=0.36]
		\begin{scope}[rotate=-90]
		\draw[line width = 1pt, color=gray] (1,0) -- (14,0);
		\draw[line width = 1pt, color=gray] (1,7) -- (14,7);
		\foreach \i in {1,...,14}{
			\draw[line width = 2pt] (\i,1) -- (\i,6);
			\draw[line width = 1pt, color=gray] (\i,0) --(\i,1);
			\draw[line width = 1pt, color=gray] (\i,6) --(\i,7);
			\draw[color=gray, fill=white] (\i,0) circle(5pt);
			\draw[color=gray, fill=white] (\i,7) circle(5pt);
		};
		\foreach \i in {1,...,6}{
			\draw[line width = 2pt] (1,\i) -- (14,\i);
		};
		\mreza{14}{6}
		\draw[color=gray, fill=gray] (1,0) circle(5pt);
		\draw[color=gray, fill=gray] (1,7) circle(5pt);
		\draw[color=gray, fill=gray] (4,0) circle(5pt);
		\draw[color=gray, fill=gray] (4,7) circle(5pt);
		\draw[color=gray, fill=gray] (9,0) circle(5pt);
		\draw[color=gray, fill=gray] (9,7) circle(5pt);
		\draw[color=gray, fill=gray] (14,0) circle(5pt);
		\draw[color=gray, fill=gray] (14,7) circle(5pt);
		\draw[fill=black] (1,1) circle(5pt);
		\draw[fill=black] (1,4) circle(5pt);
		\draw[fill=black] (1,6) circle(5pt);
		\draw[fill=black] (14,1) circle(5pt);
		\draw[fill=black] (14,4) circle(5pt);
		\end{scope}
	\end{tikzpicture}
	& &
	\begin{tikzpicture}[thick,scale=0.36]
		\draw[line width = 1.5pt, color=gray] (1,0) -- (14,0);
		\draw[line width = 1.5pt, color=gray] (1,7) -- (14,7);
		\foreach \i in {1,...,14}{
			\draw[line width = 1.5pt] (\i,1) -- (\i,6);
			\draw[line width = 1.5pt, color=gray] (\i,0) --(\i,1);
			\draw[line width = 1.5pt, color=gray] (\i,6) --(\i,7);
			\draw[color=gray, fill=white] (\i,0) circle(5pt);
			\draw[color=gray, fill=white] (\i,7) circle(5pt);
		};
		\foreach \i in {1,...,6}{
			\draw[line width = 1.5pt] (1,\i) -- (14,\i);
		};
		\mreza{14}{6}
		\draw[color=gray, fill=gray] (1,0) circle(5pt);
		\draw[color=gray, fill=gray] (1,7) circle(5pt);
		\draw[color=gray, fill=gray] (4,0) circle(5pt);
		\draw[color=gray, fill=gray] (4,7) circle(5pt);
		\draw[color=gray, fill=gray] (9,0) circle(5pt);
		\draw[color=gray, fill=gray] (9,7) circle(5pt);
		\draw[color=gray, fill=gray] (14,0) circle(5pt);
		\draw[color=gray, fill=gray] (14,7) circle(5pt);
		\draw[fill=black] (1,1) circle(5pt);
		\draw[fill=black] (1,4) circle(5pt);
		\draw[fill=black] (1,6) circle(5pt);
		\draw[fill=black] (14,1) circle(5pt);
		\draw[fill=black] (14,4) circle(5pt);
	\end{tikzpicture}
	\\
	c & & d \\
	\end{tabular}
	\caption{Strong edge geodetic set for $P_n \cp P_m$ with cardinality $\left\lceil 2 \sqrt{n} ~ \right\rceil + \left\lceil 2 \sqrt{m-2} ~ \right\rceil$ for $n=14$ and $m=8$.}
	\label{figurepropositiongeneralboundsgePnP2+sgePm-2P2}
	\end{figure}
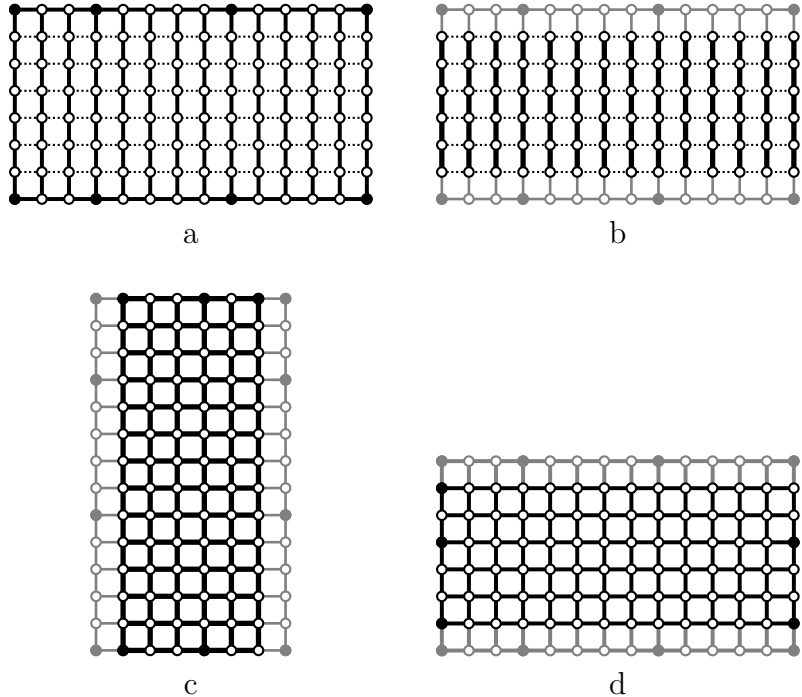
	
	In second step we look at the subgraph $H$ of $P_n \cp P_m$ without the first and the last row (Fig.~\ref{figurepropositiongeneralboundsgePnP2+sgePm-2P2}b). $H$ is isomorphic to $P_n \cp P_{m-2}$. To cover all the horizontal edges from $H$, we use Algorithm~\ref{algorithm1verticaledgesinPnPm} on $H$ rotated by $90$ degrees, which is the same as using Algorithm~\ref{algorithm1verticaledgesinPnPm} on $P_{m-2} \cp P_n$ (Fig.~\ref{figurepropositiongeneralboundsgePnP2+sgePm-2P2}c). This algorithm uses exactly $\left\lceil 2 \sqrt{m-2} ~ \right\rceil$ vertices and cover all the horizontal edges from $H$, which are exactly the edges that have not been covered in the first part of the proof.
\end{proof}

With a similar, but a bit more involved, idea as in the proof of Proposition~\ref{propositiongeneralboundsgePnP2+sgePm-2P2}, we are going to prove the following proposition.

\begin{proposition}
	\label{propositionupperboundPnPm3}
	If $n \geq 3$ and $m \geq 3$, then
	$$\sge(P_n \cp P_m) \leq \left\lceil 2 \sqrt{n+2} \right\rceil + \left\lceil 2 \sqrt{m\phantom{!}} \right\rceil - 4.$$
\end{proposition}

\begin{proof}
	First we will adjust Algorithm~\ref{algorithm1verticaledgesinPnPm} such that it will use all the corner vertices, that is $(1,1),(1,m),(n,1),(n,m)$. We call the new algorithm Algorithm~\ref{algorithm1verticaledgesinPnPm}*. For $n = k^2+h$, where $h=0$ or $k+1 \leq h \leq 2k$, Algorithm~\ref{algorithm1verticaledgesinPnPm} already uses all the corner vertices. When $n=k^2+h$, $1 \leq h \leq k$, redefine the vertex $a_k$ as $a_k = (n,1)$. The shortest paths between vertices $a_k$ and $b_i$ in Algorithm~\ref{algorithm1verticaledgesinPnPm}* are the ones that cover the same vertical edges as shortest paths between vertices $a_k$ and $b_i$ in Algortihm~\ref{algorithm1verticaledgesinPnPm}. An example output of this algorithm is shown in Fig.~\ref{figureecamplealgorithm*cornervertices}.
	
	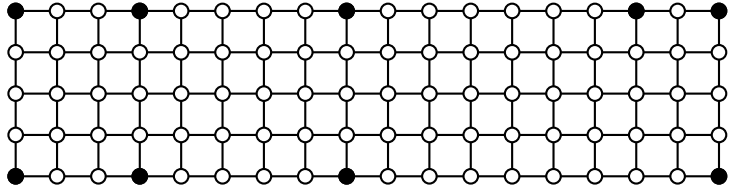
\begin{figure}[ht!]
		\centering
		\begin{tikzpicture}[thick,scale=0.55]
		\mreza{18}{5}
		\draw[fill=black] (1,1) circle(5pt);
		\draw[fill=black] (1,5) circle(5pt);
		\draw[fill=black] (4,1) circle(5pt);
		\draw[fill=black] (4,5) circle(5pt);
		\draw[fill=black] (9,1) circle(5pt);
		\draw[fill=black] (9,5) circle(5pt);
		\draw[fill=black] (16,5) circle(5pt);
		\draw[fill=black] (18,1) circle(5pt);
		\draw[fill=black] (18,5) circle(5pt);
		\end{tikzpicture}
		\caption{Strong edge geodetic set of $P_{4^2+2} \cp P_5$ from Algorithm~\ref{algorithm1verticaledgesinPnPm}*.}
		\label{figureecamplealgorithm*cornervertices}
	\end{figure}

	When $n=k^2$, $n=k^2+k-1$, $n=k^2+k$, or $n=k^2+2k$ for some $k \in \N$, we will cover the vertical edges in $P_n \cp P_m$ in the following way. First, we use Algorithm~\ref{algorithm1verticaledgesinPnPm}* with $\left\lceil  \sqrt{n} ~ \right\rceil$ vertices $V_1$. To $V_1$ we add $c=((k-1)^2+k,m)$. We can then cover the vertical edges covered by the $(1,1),(n,m)$-geodesic and the $(n,1),(1,m)$-geodesic with $(1,1),c$-geodesic and $(n,1),c$-geodesic. In this way, we can remove shortest paths between corner vertices (Fig.~\ref{figurereplacingshortestpathsaddedvertex}).

	\begin{figure}[ht!]
	  	\centering
	  	\begin{tabular}{c c c}
	  		\adjustbox{valign=c}{
	  			\begin{tikzpicture}[thick,scale=0.55]
	  			\draw[line width = 3pt] (1,1) -- (5,1) -- (5,5) -- (10,5);
	  			\mreza{10}{5}
	  			\draw[fill=black] (1,1) circle(5pt);
	  			\draw[fill=black] (1,5) circle(5pt);
	  			\draw[fill=black] (4,1) circle(5pt);
	  			\draw[fill=black] (4,5) circle(5pt);
	  			\draw[fill=black] (10,1) circle(5pt);
	  			\draw[fill=black] (9,5) circle(5pt);
	  			\draw[fill=black] (10,5) circle(5pt);
	  			\draw[color=gray, fill=gray] (7,5) circle(7pt);
	  			\end{tikzpicture}}
	  		&
	  		$\to$
	  		&
	  		\adjustbox{valign=c}{
	  			\begin{tikzpicture}[thick,scale=0.55]
	  			\draw[line width = 3pt] (1,1) -- (5,1) -- (5,5) -- (7,5);
	  			\mreza{10}{5}
	  			\draw[fill=black] (1,1) circle(5pt);
	  			\draw[fill=black] (1,5) circle(5pt);
	  			\draw[fill=black] (4,1) circle(5pt);
	  			\draw[fill=black] (4,5) circle(5pt);
	  			\draw[fill=black] (10,1) circle(5pt);
	  			\draw[fill=black] (9,5) circle(5pt);
	  			\draw[fill=black] (10,5) circle(5pt);
	  			\draw[color=gray, fill=gray] (7,5) circle(7pt);
	  			\end{tikzpicture}}
	  		\\
	  		& & \\
	  		\adjustbox{valign=c}{
	  			\begin{tikzpicture}[thick,scale=0.55]
	  			\draw[line width = 3pt] (1,5) -- (7,5) -- (7,1) -- (10,1);
	  			\mreza{10}{5}
	  			\draw[fill=black] (1,1) circle(5pt);
	  			\draw[fill=black] (1,5) circle(5pt);
	  			\draw[fill=black] (4,1) circle(5pt);
	  			\draw[fill=black] (4,5) circle(5pt);
	  			\draw[fill=black] (10,1) circle(5pt);
	  			\draw[fill=black] (9,5) circle(5pt);
	  			\draw[fill=black] (10,5) circle(5pt);
	  			\draw[color=gray, fill=gray] (7,5) circle(7pt);
	  			\end{tikzpicture}}
	  		&
	  		$\to$
	  		&
	  		\adjustbox{valign=c}{
	  			\begin{tikzpicture}[thick,scale=0.55]
	  			\draw[line width = 3pt] (7,5) --(7,1) -- (10,1);
	  			\mreza{10}{5}
	  			\draw[fill=black] (1,1) circle(5pt);
	  			\draw[fill=black] (1,5) circle(5pt);
	  			\draw[fill=black] (4,1) circle(5pt);
	  			\draw[fill=black] (4,5) circle(5pt);
	  			\draw[fill=black] (10,1) circle(5pt);
	  			\draw[fill=black] (9,5) circle(5pt);
	  			\draw[fill=black] (10,5) circle(5pt);
	  			\draw[color=gray, fill=gray] (7,5) circle(7pt);
	  			\end{tikzpicture}}
	  	\end{tabular}
	  	\caption{Covering vertical edges in $P_n \cp P_m$ without using shortest paths between vertices $(1,1)$ and $(n,m)$ and between vertices $(n,1)$ and $(1,m)$.}
	  	\label{figurereplacingshortestpathsaddedvertex}
	  \end{figure}
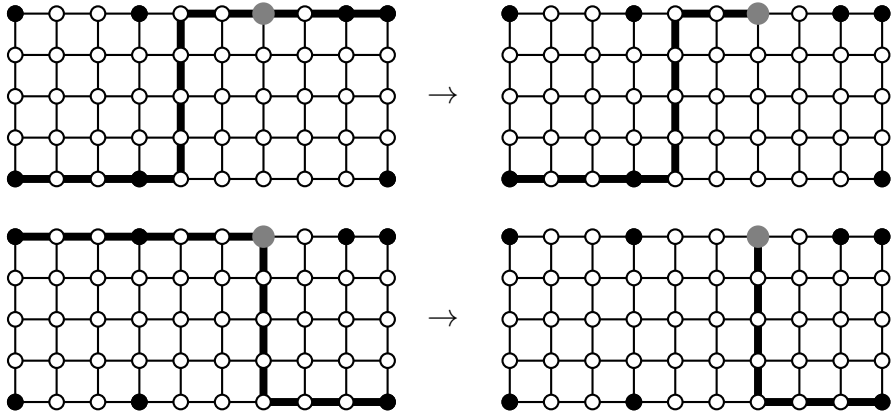

	When $n = k^2 + h$; $k, h \in \N$; $1 \leq h \leq k-2$, we will adjust Algorithm~\ref{algorithm1verticaledgesinPnPm}* to cover all the vertical edges with the same vertex set and without using the shortest paths between vertices $a_1$ and $b_{k+1}$ and between vertices $a_k$ and $b_1$. First we notice that Algorithm~\ref{algorithm1verticaledgesinPnPm}* does not use the unique $a_k$,$b_{k+1}$-geodesic. If we add it, we can replace the $a_{h}$,$b_{k+1}$-geodesic with the one that covers all the vertical edges in the $(k^2+1)$-th column. We can then remove the $a_1$,$b_{k+1}$-geodesic (it also covers the vertical edges in the $(k^2+1)$-th column). Algorithm~\ref{algorithm1verticaledgesinPnPm}* also does not use the $a_{k-1}$,$b_{k+1}$-geodesic. If we add the $a_{k-1}$,$b_{k+1}$-geodesic that covers all the vertical edges in the $((k-1)^2+k)$-th column, we can remove the $a_k$,$b_{k+1}$-geodesic.
	
	When $n = k^2 + h$, $k, h \in \N$, $k+1 \leq h \leq 2k-1$, we can use a similar adjustment as in the proof of Theorem~\ref{theoremsgePnP4} for $n=k^2+h$, $k+1 \leq h \leq 2k-1$, and remove the shortest paths between vertices $a_1$ and $b_{k+1}$ and between vertices $a_{k+1}$ and $b_1$ that cover the horizontal edges in the second and the third row.
	
	In the first part of the proof we defined the algorithm that uses $\left\lceil 2 \sqrt{n+2} ~ \right\rceil$ vertices and covers all the vertical rows in $P_n \cp P_m$, while it uses neither the $(1,1)$,$(n,m)$-geodesic, nor the $(1,m)$,$(n,1)$-geodesic.
	
	To cover the horizontal edges, we use Algorithm~\ref{algorithm1verticaledgesinPnPm}* on rotated $P_n \cp P_m$ by $90$ degrees. This part will use $\left\lceil  \sqrt{m} ~ \right\rceil$ vertices, where four ($(1,1),(1,m),(n,1),(n,m)$) of them are already used in the first part. Because the first part does not use shortest paths between these four vertices, the condition that between any two vertices we use at most one shortest path still holds.
	
	In the first and in the second part we have covered all the edges in $P_n \cp P_m$, using exactly $\left\lceil  \sqrt{n+2} ~ \right\rceil + \left\lceil  \sqrt{m} ~ \right\rceil - 4$ vertices.  
\end{proof}

If we combine all the results from this paper, we can see that the bound from Proposition~\ref{propositiongeneralboundsgePnP2+sgePm-2P2} is sharp when $m=2$ and is not sharp for $m=3$ and $m=4$. The bounds from Proposition~\ref{propositionupperboundPnPm3} are sharp when $m=3$ and $n=k^2$ or $k^2+k$ for some integer $k$, as well as when $m=4$ and $n=k^2$, $k^2+k$ or $k^2+2k$ for some integer $k$.

\medskip

\bibliography{references.bib}
\bibliographystyle{plain}

\end{document}